\documentclass[onefignum,onetabnum]{siamart190516}

\usepackage{lipsum}
\usepackage{amsfonts}
\usepackage{graphicx}
\usepackage{epstopdf}
\usepackage{algorithmic}
\usepackage{mathrsfs}
\usepackage{bm}
\usepackage{algorithm}
\usepackage{float}
\usepackage{comment}

\newcommand{\E}{\mathbb{E}}

\newcommand{\R}{\mathbb{R}}

\newcommand{\x}{\boldsymbol{x}}

\newcommand{\uu}{\boldsymbol{u}}

\ifpdf
  \DeclareGraphicsExtensions{.eps,.pdf,.png,.jpg}
\else
  \DeclareGraphicsExtensions{.eps}
\fi

\usepackage{enumitem}
\setlist[enumerate]{leftmargin=.5in}
\setlist[itemize]{leftmargin=.5in}

\newcommand{\eps}{\varepsilon}
\newsiamremark{Hypothesis}{Hypothesis}
\crefname{hypothesis}{Hypothesis}{Hypotheses}
\newsiamthm{Claim}{Claim}
\newsiamremark{remark}{Remark}
\newsiamremark{Definition}{Definition}
\newsiamthm{Corollary}{Corollary}
\newsiamthm{assumption}{Assumption}
\newsiamthm{Theorem}{Theorem}
\newsiamthm{Proposition}{Proposition}
\newsiamthm{Lemma}{Lemma}

\headers{QMC for SDE Simulation}{Du Ouyang, Zexin Pan, and Zhijian He}

\title{Quasi-Monte Carlo for SDE Simulation: Error Analysis and Dimensionality Reduction\thanks{Submitted to the editors DATE.
\funding{}}}

\author{Du Ouyang\thanks{Department of Mathematical Sciences, Tsinghua University 
  (\email{duouyang99@outlook.com}).  } \and Zexin Pan\thanks{Institute of Fundamental and Transdisciplinary Research, Zhejiang University, Zhejiang 310058, People's Repulic of China
  (\email{zep002@zju.edu.cn}).} \and  Zhijian He\thanks{Corresponding author. School of Mathematics, South China University of Technology, Guangzhou 510641, People's Repulic of China (\email{hezhijian@scut.edu.cn}).}
}

\usepackage{amsopn}

\ifpdf
\hypersetup{
  pdftitle={QMC for SDE Simulation},
  pdfauthor={D. Ouyang, Z. Pan, and Z. He}
}
\fi

\newcommand{\dd}{\mathrm{d}}

\newcommand{\p}{\mathbb{P}}
\newcommand{\1}{\mathbf{1}}

\begin{document}

\maketitle

% REQUIRED
\begin{abstract}
We investigate the numerical simulation of general stochastic differential equations (SDEs) using Quasi-Monte Carlo (QMC) methods. First, we provide a rigorous theoretical analysis of the QMC method applied to the Euler-Maruyama (EM) scheme, establishing that it significantly accelerates the decay of the sampling error and achieves an asymptotically superior convergence rate over the classical Monte Carlo method. Second, the traditional EM scheme exhibits a slow polynomial decay of the discretization error, which necessitates a large number of time steps and leads to a significantly high integration dimension. To address this issue, we propose a Multilevel Stochastic Time Grid (MSTG) method based on Exact Simulation techniques, and we rigorously establish its convergence rate under randomized QMC sampling, proving that it preserves the high-order convergence of the sampling error. In terms of the overall error, the truncation error of the proposed MSTG method exhibits a remarkably fast super-exponential decay. Consequently, to achieve a given accuracy level, our approach requires significantly fewer discretization steps than the EM scheme, thereby drastically reducing the actual integration dimension of the QMC method. This substantial dimensionality reduction strategy greatly enhances the practical efficiency of the QMC algorithm. Numerical experiments fully corroborate the superiority of the proposed approach.
\end{abstract}

% REQUIRED
\begin{keywords}
  Stochastic differential equations, Quasi-Monte Carlo, Exact simulation, Dimension reduction
\end{keywords}

% REQUIRED
\begin{AMS}
  
\end{AMS}

\section{Introduction}
\label{sec:sde_qmc_review}

Stochastic differential equations (SDEs) serve as the standard mathematical language for characterizing the evolution of continuous-time dynamical systems driven by random noise. Since the vast majority of nonlinear SDEs lack exact analytical solutions, their numerical approximations are indispensable for interdisciplinary applications. This field serves as the theoretical cornerstone for classical stochastic optimal control and computational finance (e.g., derivative pricing, risk management \cite{glasserman2004,YZbook}). Furthermore, it has been broadly applied in statistical mechanics, computational biology, and rapidly developing fields of artificial intelligence, such as generative diffusion models \cite{song2021ScoreBased} and continuous-time reinforcement learning \cite{wang2020a}. To systematically elucidate the construction of numerical methods, we consider the following general $d$-dimensional SDE:
\begin{equation}
\mathrm{d}X_t = b(t, X_t)\mathrm{d}t + \sigma(t, X_t)\mathrm{d}W_t, \quad t \in [0, T].
\end{equation}

In practical scientific computing, the primary objective is typically to evaluate the expectation of a specific functional of the terminal system state, denoted as $\mathbb{E}[g(X_T)]$. Traditionally, this expectation is approximated using time-grid-based discretization schemes, such as the Euler-Maruyama (EM) method \cite{kloeden1992numerical}, combined with Monte Carlo (MC) simulation. However, the overall Root Mean Square Error (RMSE) of the EM-MC estimator is bounded by $\mathcal{O}(N^{-1}) + \mathcal{O}( m^{-1/2})$, where $N$ is the number of time steps and $m$ is the sample size. Achieving high accuracy necessitates massive sample sizes and highly refined time grids, leading to prohibitive computational costs.

To accelerate convergence, a natural idea is to apply the quasi-Monte Carlo (QMC) method to SDE simulations by substituting random sequences with low-discrepancy sequences to generate Brownian motion increments. However, existing research regarding QMC applications in SDEs (e.g., \cite{giles2009Multilevela}) predominantly focuses on empirical analysis, with scant literature rigorously addressing the convergence rate of QMC for the SDE simulation.

In practice, each simulation step of an SDE requires mapping uniformly distributed samples to normally distributed ones via the inverse cumulative distribution function. This implies that the integrand is inherently unbounded at the boundaries of the integration domain. The classical Koksma-Hlawka inequality requires the integrand to have Bounded Variation in the sense of Hardy and Krause (BVHK), a condition typically violated in SDE simulations. The results in \cite{owen2006a,ouyang2024achieving} regarding QMC integration for unbounded functions precisely bridges this theoretical gap. Based on these results, we proved that, provided the integrand (here, the numerical solution of the SDE parameterized by the input random variables) satisfies specific growth conditions, the QMC method can still preserve the convergence rate of $ \mathcal{O}(m^{-1+\eps})$ for SDE estimation (see Section \ref{sec: qmc euler}), where $ m$ is the number of samples and $ \eps>0$ can be arbitrarily small. This provides a solid theoretical guarantee for applying QMC to the numerical simulation of SDEs.

Nevertheless, applying QMC to SDE simulations faces a formidable practical obstacle: the contradiction between the required discretization accuracy and the high-dimensional challenge. To reduce the $\mathcal{O}(N^{-1})$ EM discretization bias, $N$ must be significantly increased. This severely inflates the nominal integration dimension $d_{\text{nom}} = d \times N$, which drastically degrades QMC efficiency. Extensive research has addressed this high-dimensional challenge. Tractability theories in weighted spaces \cite{sloan1998When, kuo2016Application} explain QMC's efficacy, though they often exclude non-smooth financial payoffs. Alternatively, path generation techniques—such as Brownian Bridge and Principal Component Analysis \cite{caflisch1997Valuation, wang2003effectivea, wang2006Effects, wang2009Dimension, wang2010QuasiMonte, he2014Good}—reduce the \textit{effective dimension} by concentrating variance in the primary variables. Furthermore, methods aligning discontinuity surfaces \cite{wang2012Pricing, wang2016Handling} have been proposed to handle non-smooth payoffs. Yet, regardless of how effectively the dimension is reduced, these techniques cannot eliminate the underlying systematic bias imposed by the EM scheme.

To completely eradicate the discretization bias and the resulting high-dimensional burden, Exact Simulation techniques have emerged. Early breakthroughs utilized rejection sampling \cite{beskos2005Exact, beskos2006Retrospective}, though their acceptance rates decay exponentially in high dimensions. To overcome this, randomized truncation methods based on Multilevel MC (MLMC) were introduced \cite{mcleish2011general, rhee2015Unbiased}, utilizing geometrically distributed random variables to achieve unbiased estimation via telescoping sums, sparking numerous subsequent studies \cite{vihola2018Unbiased}.

Another highly effective Exact Simulation framework relies on operator expansion and Stochastic Time Grids \cite{henry-labordere2017Unbiased, andersson2017Unbiased, doumbia2017Unbiased}. By introducing a Poisson process, the deterministic time integration associated with the Feynman-Kac PDE is ingeniously converted into an expectation over random jump times. This construction naturally eliminates the need for deterministic time discretization, achieving unbiased estimation of SDE functionals.

Despite their tremendous success in eliminating bias, these Exact Simulation techniques are fundamentally built upon standard MC sampling. Integrating QMC into this unbiased domain introduces a dual challenge: Unbiased estimators frequently incorporate indicator functions, causing severe discontinuities with respect to the underlying random variables, which can drastically degrade the convergence rate of QMC \cite{he2015convergence}.

In conclusion, how to overcome the challenges of variable dimensions and discontinuities within the Exact Simulation framework of \cite{henry-labordere2017Unbiased} to design an efficient QMC sampling algorithm remains an unresolved problem with significant theoretical and practical implications. This is precisely the primary focus of the present paper.

\section{Preliminaries}

To facilitate the subsequent exposition, we introduce the following notations. Denote $ 1{:}N = \left\{ 1,\dots,N\right\} $. Consider a subset of coordinate indices $\boldsymbol{u} \subseteq 1{:}N$. Let $ |\bm u|$ be the number of elements in $ \bm u$.  Denote $ \partial^{\bm u}_{\bm x} f:= (\prod_{j\in \bm u}\partial_{x_j}) f$. 

\begin{definition}
    We define the function class $\mathcal{C}_p(\R^d;\R^m)$ as the set of all functions satisfying the following property: a function $f(\bm x):\R^d \rightarrow \R^m$ belongs to $\mathcal{C}_p(\R^d;\R^m)$ if and only if $|f|$ can be bounded by a polynomial in $|\bm x|$; that is, there exist constants $M,K,B>0$ such that:
    \begin{equation}
        |f(\bm x)| \le M|\bm x|^K+B.
    \end{equation}
\end{definition}

By definition, the class $\mathcal{C}_p(\R^d;\R^m)$ comprises all functions whose growth is polynomially bounded. When $m = 1$, we simplify the notation to $\mathcal{C}_p(\R^d)$.

\begin{definition}\label{def:smooth}
    
Let $K$ be $[0,1]^d$ or $\R^d$. A function $f(\x)$ defined over $K$ is called a \textbf{smooth function} if for any $ \uu \subseteq 1{:}d$, $ \partial^{\uu} f $ is continuous. Let $ \mathcal{S}^d(K)$ be the class of such smooth functions.
\end{definition}

Recalling the function classes introduced by \cite{ouyang2024achieving}.
\begin{definition}
    For $ M,B,K>0$, define the polynomial growth class as
    \begin{equation}\notag
        G_p(M,B,K):= \left\{h \in \mathcal{S}^d(\R^d)  :\sup_{\uu \subseteq 1{:}d} \left|\partial^{\boldsymbol{u}}h(\boldsymbol{x})\right| \le M|\x|^K+B\ \text{for all}\  \boldsymbol{x} \in \R^d\right\}.
    \end{equation}
\end{definition}

We say that a function $F$ has \textbf{polynomial growth} if there exist constants $M,B,K>0$ such that $F\in G_p(M,B,K)$.

We next recall some standard notions from QMC theory; see, for example, \cite{Niederreiter1992,dick2010}. For a point set $\mathcal{P}_m=\{\boldsymbol{y}_1,\dots,\boldsymbol{y}_m\}\subset[0,1]^s$, its star discrepancy is defined as follows \cite{Niederreiter1992,dick2010}:
\begin{equation}
    D_m^*(\mathcal{P}_m)
    :=
    \sup_{\boldsymbol{a}\in[0,1]^s}
    \left|
        \frac{1}{m}\sum_{j=1}^m \mathbf{1}_{[\boldsymbol{0},\boldsymbol{a})}(\boldsymbol{y}_j)
        - \prod_{\ell=1}^s a_\ell
    \right|.
\end{equation}
If $f$ has bounded variation in the sense of Hardy and Krause, denoted by $V_{\mathrm{HK}}(f)$, then the Koksma-Hlawka inequality gives the bound \cite{Hlawka1972,Niederreiter1992,dick2010}:
\begin{equation}
    \left|
        \frac{1}{m}\sum_{j=1}^m f(\boldsymbol{y}_j)
        -\int_{[0,1]^s} f(\boldsymbol{y})\,\dd \boldsymbol{y}
    \right|
    \le V_{\mathrm{HK}}(f)D_m^*(\mathcal{P}_m).
\end{equation}
Thus the efficiency of QMC rules is governed by the uniformity of the point set. A point set is called a low-discrepancy point set if $D_m^*(\mathcal{P}_m)=\mathcal{O}(m^{-1}(\log m)^{s-1})$, while an infinite sequence is called a low-discrepancy sequence if its first $m$ points satisfy $D_m^*=\mathcal{O}(m^{-1}(\log m)^s)$ \cite{Niederreiter1992,dick2010}.

A classical construction of low-discrepancy points is provided by digital nets and sequences \cite{Niederreiter1992,dick2010}. In this paragraph only, we follow the classical QMC notation: $m$ denotes the net parameter and $s$ denotes the integration dimension; these symbols may have different meanings in the later SDE analysis. Let $b\ge2$ be an integer base. An elementary interval in $[0,1)^s$ is a set of the form
\[
    E=\prod_{j=1}^s \left[\frac{a_j}{b^{c_j}},\frac{a_j+1}{b^{c_j}}\right),
\]
where $c_j\ge0$ and $0\le a_j<b^{c_j}$ are integers. A point set with $b^m$ points in $[0,1)^s$ is called a $(t,m,s)$-net in base $b$ if every elementary interval $E$ with volume $b^{t-m}$ contains exactly $b^t$ points. An infinite sequence $\{\boldsymbol y_0,\boldsymbol y_1,\dots\}$ is called a $(t,s)$-sequence in base $b$ if, for every $k\ge0$ and every $m>t$, the block
\[
    \boldsymbol{y}_{kb^m},\boldsymbol{y}_{kb^m+1},\dots,\boldsymbol{y}_{(k+1)b^m-1}
\]
forms a $(t,m,s)$-net in base $b$. Such nets and sequences satisfy the low-discrepancy bounds above \cite{Niederreiter1992,dick2010}.

Randomized QMC (RQMC) methods introduce randomness into deterministic low-discrepancy point sets while preserving their uniformity structure. We use the following general definition.

\begin{definition}[RQMC point sets and sequences]\label{def:rqmc point set}
    Let $s,m\ge 1$, and let $\mathcal{P}_m=\{\boldsymbol{y}_1,\dots,\boldsymbol{y}_m\}$ be a random point set in $[0,1]^s$. We call $\mathcal{P}_m$ a \textbf{RQMC point set} if it satisfies
    \begin{enumerate}[label=(\roman*)]
        \item for each $j=1,\dots,m$, the random point $\boldsymbol{y}_j$ is uniformly distributed on $[0,1]^s$, i.e.,
        $\boldsymbol{y}_j \sim \mathcal{U}([0,1]^s)$;
        \item the point set almost surely preserves the low-discrepancy property, i.e., there exists a constant $C_s>0$ independent of $m$ such that
        \[
            D_m^*(\mathcal{P}_m)
            \le C_s\frac{(\log m)^{s-1}}{m},
            \quad \mathrm{a.s.},
        \]
    \end{enumerate}
    where $D_m^*$ denotes the star discrepancy.

    Similarly, let $\{\boldsymbol{y}_1,\boldsymbol{y}_2,\dots\}$ be a random sequence in $[0,1]^s$. We call it a \textbf{RQMC sequence} if it satisfies
    \begin{enumerate}[label=(\roman*)]
        \item for each $j\ge 1$, the random point $\boldsymbol{y}_j$ is uniformly distributed on $[0,1]^s$, i.e.,
        $\boldsymbol{y}_j \sim \mathcal{U}([0,1]^s)$;
        \item the sequence almost surely preserves the low-discrepancy property, i.e., there exists a constant $C_s>0$ independent of $n$ such that, for every $n\ge 2$,
        \[
            D_n^*(\{\boldsymbol{y}_1,\dots,\boldsymbol{y}_n\})
            \le C_s\frac{(\log n)^s}{n},
            \quad \mathrm{a.s.},
        \]
    \end{enumerate}
\end{definition}

The first condition in Definition \ref{def:rqmc point set} ensures that the RQMC estimator is unbiased and can therefore be analyzed through variance or RMSE. The second condition guarantees that the randomized point set or sequence inherits the low-discrepancy property of the original deterministic construction. Common randomizations include random shifts for lattice rules \cite{kuo2006a,kuo2010} and digital randomizations for digital nets and sequences. In particular, Owen's scrambling preserves the digital-net structure almost surely and makes each randomized point uniformly distributed on the unit cube \cite{owen1995,owen1997b,dick2010}. Therefore, a scrambled $(t,m,s)$-net gives an RQMC point set, while a scrambled $(t,s)$-sequence, such as an Owen-scrambled Sobol' sequence, gives an RQMC sequence. For sufficiently smooth integrands, scrambled nets can even yield higher-order RMSE convergence rates; see \cite{owen1997b,owen2008,dick2010}. In this paper, we only use the two general properties stated in Definition \ref{def:rqmc point set}.

\section{QMC Error for the Euler-Maruyama Scheme}\label{sec: qmc euler}

Since SDEs generally lack exact analytical solutions, the EM scheme is commonly employed to simulate trajectories, thereby yielding an estimator with an $\mathcal{O}(N^{-1})$ discretization error based on $N$ grid nodes. To be precise, consider an SDE defined on a probability space $(\Omega, \mathcal{F}, \mathbb{P})$:
\begin{equation}
    \dd X_t = b(t,X_t)\dd t + \sigma(t,X_t) \dd W_t, \quad X_0 = x_0,\quad t\in[0,T],
\end{equation}
where $b:[0,T]\times\R^d \to \R^d$, $\sigma: [0,T]\times\R^d \to \R^{d\times d}$, and $W$ is a standard $d$-dimensional Brownian motion. Our objective is to numerically compute the expectation $\E\left[ g(X_T)\right]$. A standard approach is to utilize the EM scheme. For a given time grid $\mathscr{G}:= \{0<t_0<\cdots<t_N = T\}$, we simulate $m$ Brownian motion trajectories $\{W^{(k)}_{t_i}: 1\le i\le N,\ 1\le k \le m\}$ at the grid nodes, and iteratively generate the approximations $\{X_{t_i}^{(k)}\}_{i=1}^N$ as follows:
\begin{equation}
    X^{(k)}_{t_i} = X^{(k)}_{t_{i-1}} + b(t_{i-1},X^{(k)}_{t_{i-1}})\Delta t_i + \sigma(t_{i-1},X^{(k)}_{t_{i-1}})\Delta W^{(k)}_{t_i},
\end{equation}
where $\Delta t_i := t_i - t_{i-1}$ and $\Delta W^{(k)}_{t_i} = W^{(k)}_{t_i} - W^{(k)}_{t_{i-1}}$. Subsequently, the sample mean is used as a biased estimator to approximate $\E\left[ g(X_T)\right]$, namely:
\begin{equation}\label{eq:euler}
    \tilde I_{m,N}^{\textrm{MC}}: = \frac{1}{m}\sum_{k = 1}^m g(X^{(k)}_{t_N}) \approx \E\left[ g(X_T)\right].
\end{equation}

According to the results in \cite{kloeden1992numerical}, under some regularity conditions of $ b$, $ \sigma$ and $ g$, the EM scheme possesses a first-order weak convergence rate with respect to time discretization. For the sake of simplicity, assuming a uniform grid $t_i = iT/N$, we have:
\begin{equation}
    \left|\E\left[ g(X^{(k)}_{t_N})\right] - \E\left[ g(X_T)\right]\right| \le C\frac{1}{N}.
\end{equation}
Therefore, the RMSE of the estimator $\tilde I_{m,N}$ has the following bound.
\begin{lemma}\label{lemma: em mc rmse}
Assume that $ b $ and $ \sigma$ satisfy the linear growth condition, i.e., there is a constant $ C>0$ such that 
\begin{equation}\label{eq: linear growth condi}
    |b(t,x)| + |\sigma(t,x)| \le C(1 + |x|).
\end{equation}
Moreover, assume that $ b$, $ \sigma$ and $ g$ satisfy the regularity conditions \cite{kloeden1992numerical} so that the EM scheme admits the first order weak convergence. Then, there is a constant $ C$ such that
\begin{equation}\label{eq:rmse of mc}
    \sqrt{\E\left[ (\tilde I_{m,N}^{\textrm{MC}} - \E\left[ g(X_T)\right])^2\right]} \le C\left(\frac{1}{N} + \frac{1}{\sqrt{m}}\right).
\end{equation}
\end{lemma}

Based on \eqref{eq:rmse of mc}, it is evident that the RMSE is governed by both the discretization error and the sampling error. To reduce the RMSE and improve computational efficiency, we seek to elevate the convergence rate of the sampling error. According to classical MC theory, the sampling error decays at a rate of $\mathcal{O}(m^{-1/2})$. A natural idea is to enhance this rate by incorporating the QMC method.

In a related study, \cite{giles2009Multilevela} incorporated the QMC method—specifically rank-1 lattice rules combined with Brownian bridge construction—into the Multilevel Monte Carlo (MLMC) framework, applying it alongside EM and Milstein discretization schemes for financial derivative pricing. Their numerical results demonstrated that this approach can significantly boost computational accuracy compared to traditional MLMC. However, their research primarily focuses on algorithm construction and empirical demonstrations of numerical performance; it has yet to provide a rigorous mathematical proof for the error bounds of the constructed multilevel QMC estimator.

In this section, we aim to establish a rigorous QMC theory and formally prove the convergence rate of the sampling error for the EM scheme under the QMC framework. Notably, because the construction of the EM scheme entails simulating unbounded Brownian motion increments, the resulting integrand is inherently unbounded, rendering classical QMC theory invalid. To this end, we will use the results established in \cite{ouyang2024achieving} for the convergence rate of the QMC method with unbounded integrands. This analytical framework can be directly applied to the sampling error of the EM scheme.

For simplicity of exposition, we focus on a \textbf{one-dimensional} stochastic differential equation. The theoretical conclusions readily extend to higher-dimensional cases, albeit with more cumbersome mathematical notation. 

To explicitly analyze the dependence of the numerical solution on the underlying driving noise, we formulate the EM scheme as a deterministic function of the input vector $\bm z:=(z_1,\dots,z_N)^\top \in \mathbb{R}^N$:
\begin{equation}\label{eq: EM function X}
    \widetilde X_i(\bm z) = \widetilde X_{i-1}(\bm z) + b(t_{i-1}, \widetilde X_{i-1}(\bm z))\Delta t + \sigma(t_{i-1}, \widetilde X_{i-1}(\bm z))\sqrt{\Delta t} z_i, \quad i=1, \dots, N,
\end{equation}
where $\widetilde X_0 = x_0$ and $\Delta t = T/N$. It is worth noting that when the inputs $z_i$ are sampled as independent standard normal random variables (i.e., $z_i \overset{i.i.d.}{\sim} \mathcal{N}(0, 1)$), this functional representation naturally recovers the standard MC version of the EM scheme.

Consequently, the quantity of interest at the terminal time can be expressed directly as the composite functional $g(\widetilde X_N(\bm z))$. Based on the definition of the polynomial growth condition, to establish the QMC error bound, it suffices to prove that for any set $\boldsymbol{u} \subseteq 1{:}N$, the mixed partial derivative of this functional belongs to the class:
\begin{equation}
    \partial^{\bm u}_{\bm z} g(\widetilde X_N(\bm z)) \in \mathcal{C}_p(\R^N).
\end{equation}

To this end, we need to specify the smoothness conditions for the coefficients, which are detailed in the following assumption.

\begin{assumption}\label{assum: qmc euler coef}
    For any $t\in[0,T]$, the functions $b(t,\cdot), \sigma(t,\cdot)$, and $g$ are $N$-times continuously differentiable. Moreover, for any $0\le k\le N$, $\sup_{t\in[0,T]}|\partial_x^k b(t,x)|$, $\sup_{t\in[0,T]}|\partial_x^k \sigma(t,x)|, \frac{\dd^k}{\dd x^k } g \in \mathcal{C}_p(\R)$. Furthermore, we assume that $b$ and $\sigma$ satisfy the linear growth condition (see \eqref{eq: linear growth condi})
\end{assumption}

\begin{remark}
    Note that the above assumption is specifically designed to ensure that the integrand satisfies the smoothness and polynomial growth conditions required for QMC estimation. Deriving the actual convergence rate results (both strong and weak convergence) of the EM scheme necessitates additional assumptions on the coefficients (see \cite{kloeden1992numerical}).
\end{remark}

We will first derive the growth properties of the derivatives of $\widetilde X_i$. The following theorem constitutes a key conclusion of this section.

\begin{theorem}\label{thm: partial deri growth}
    If Assumption \ref{assum: qmc euler coef} holds, then there exist constants $M,B,K>0$ such that for any $0\le i \le N$,
    \begin{equation}
        \widetilde X_i \in G_p(M,B,K).
    \end{equation}
\end{theorem}

\begin{proof}
    See Section \ref{sec proof thm: partial deri growth}.
\end{proof}

Theorem \ref{thm: partial deri growth} characterizes the growth property of the mixed partial derivatives $\partial^{\boldsymbol{u}}_{\bm z} \widetilde X_i$ with respect to the Brownian motion increments $\bm z$. This theorem indicates that, as a function of $\bm z$, $\widetilde X_i$ exhibits polynomial growth. Building upon this conclusion, we can rigorously establish the convergence rate of the EM scheme under the QMC method.

Suppose $\mathcal{P}_m = \{\boldsymbol{y}_1, \dots, \boldsymbol{y}_m\}$ is a low-discrepancy point set on the $N$-dimensional unit cube $[0,1]^N$. Consider the following RQMC estimator:
\begin{equation}
    \tilde I_{m,N}^{\textrm{RQMC}} := \frac{1}{m}\sum_{k=1}^m g(\widetilde X_N(\bm z^{(k)})).
\end{equation}
Here, the function $ \widetilde X_N$ is defined by \eqref{eq: EM function X}, and $ \bm z^{(k)} = (z^{(k)}_1, \dots, z^{(k)}_N)^\top = \Phi^{-1}(\boldsymbol{y}_k)$, where $ \Phi^{-1}$ is the inverse cumulated distribution function of standard Gaussian distribution acting on each component of $ \bm y_k$. 

The following theorem provides the RMSE upper bound for the RQMC estimator $\tilde I_{m,N}^{\textrm{RQMC}}$ with respect to the target value $\E[g(X_T)]$.

\begin{theorem}
    \label{thm: qmc error}
    Suppose Assumption \ref{assum: qmc euler coef} holds, and $b, \sigma, g$ satisfy the regularity conditions ensuring the first-order weak convergence of the EM scheme. If $\{\boldsymbol{y}_1, \dots, \boldsymbol{y}_m\}$ is an RQMC point set, then there are constants $ C(\eps,N), C > 0$ such that for any $ \eps > 0$,
    \begin{equation}
        \sqrt{\E\left[ \left| \tilde I_{m,N}^{\textrm{RQMC}} - \E[g(X_T)]\right|^2\right]} \le  C(\eps,N) m^{-1+\eps} + C \frac{1}{N}.
    \end{equation}
\end{theorem}

\begin{proof}
    Utilizing the triangle inequality, we decompose the overall error into the sampling error and the discretization error:
    \begin{equation}
        \begin{aligned}
        &\sqrt{\E\left[ \left| \tilde I_{m,N}^{\textrm{RQMC}} - \E[g(X_T)]\right|^2\right]}\\
        &\le \sqrt{\E\left[ \left|\tilde I_{m,N}^{\textrm{RQMC}} - \E[g(\widetilde X_N)] \right|^2\right]} + \left| \E[g(\widetilde X_N)] - \E[g(X_T)]\right|.
        \end{aligned}
    \end{equation}
    According to the weak convergence theory for numerical solutions of SDEs (see \cite{kloeden1992numerical}), the upper bound of the second term (i.e., the discretization error) is $\mathcal{O}(N^{-1})$. Therefore, it suffices to focus on bounding the first term (i.e., the sampling error).
    
    By \cite[Corollary 4.8]{ouyang2024achieving}, to obtain the convergence rate of RQMC methods, it suffices to prove that the integrand $g(\widetilde X_N(\bm z))$ has polynomial growth. For any index set $\boldsymbol{u} \subseteq 1{:}N$, according to Fa\`a di Bruno's formula, we have:
    \begin{equation}\label{eq: partial g(x)}
        \partial^{\boldsymbol{u}} g(\widetilde X_N) = \sum_{w \in \pi(\boldsymbol{u})} \frac{\dd^{|w|} }{\dd x^{|w|}}g(\widetilde X_N) \prod_{\lambda \in w} \partial^{\lambda} \widetilde X_N.
    \end{equation}
    From Assumption \ref{assum: qmc euler coef}, the derivatives of the terminal function $g$ satisfy $\frac{\dd^k }{\dd x^k}g\in \mathcal{C}_p(\R), 1\le k\le N$. Meanwhile, based on the conclusions of Theorem \ref{thm: partial deri growth}, both $\widetilde X_N$ itself and its mixed partial derivatives $\partial^{\lambda} \widetilde X_N$ with respect to $\bm z$ belong to the function class $\mathcal{C}_p(\R^N)$. Consequently, combining this with \eqref{eq: partial g(x)} yields $\partial^{\boldsymbol{u}}_{\bm z} g\circ\widetilde X_N \in \mathcal{C}_p(\R^N)$. Thus, there exist constants $M, B, K > 0$ such that for any $\boldsymbol{u} \subseteq 1{:}N$, it holds that
    \begin{equation}
        |\partial^{\boldsymbol{u}} g(\widetilde X_N(\bm z))| \le M|\bm z|^K + B.
    \end{equation}
    This confirms that the integrand falls into the polynomial growth class $g(\widetilde X_N) \in G_p(M, B, k)$. A direct application of \cite[Corollary 4.8]{ouyang2024achieving} establishes the convergence rate of $\mathcal{O}(m^{-1+\eps})$. This completes the proof.
\end{proof}

By contrasting Theorem \ref{thm: qmc error} with the MC results in Lemma \ref{lemma: em mc rmse} , we can distinctly observe both the advantages and the potential challenges of applying the RQMC method to SDE simulations.

\paragraph{Advantage in Convergence Rate} 
For a fixed number of time discretization steps $N$, the statistical error of the MC method decays at a rate of $\mathcal{O}(m^{-1/2})$, whereas the RQMC method achieves an accelerated convergence rate of $\mathcal{O}(m^{-1+\eps})$. This implies that to attain a given statistical accuracy $\delta$, the sample size $m$ required by the MC method must be proportional to $\delta^{-2}$, while the RQMC method merely requires a sample size proportional to $\delta^{-1}$. In applications demanding high precision, such as financial derivative pricing, the RQMC method can significantly diminish the requisite number of simulated trajectories, thereby drastically boosting computational efficiency.

\paragraph{Dimensionality Dependence and Challenges}
Although RQMC exhibits an asymptotic superiority with respect to the sample size $m$, it is imperative to acknowledge that the constant term $C(\eps,N)$ in Theorem \ref{thm: qmc error} heavily depends on the problem's dimension, namely the number of time steps $N$. In SDE simulations, mitigating the discretization error $\mathcal{O}(N^{-1})$ typically necessitates an extensively large $N$. However, as $N$ increases, the actual integration dimension escalates. Since the error bound of RQMC is inherently tied to the integration dimension, an excessively high dimension will degrade the uniformity of the point sets in high dimensions, thus severely attenuating the practical efficiency of the RQMC method.

Consequently, while directly applying RQMC to SDE simulations yields a superior theoretical convergence rate, an excessively large $N$ results in an unmanageably high integration dimension that undermines the practical efficacy of RQMC. This high-dimensional challenge precisely serves as the motivation for introducing the Stochastic Time Grid estimator based on Exact Simulation in the subsequent portions of this chapter: by eliminating or substantially mitigating the reliance on a large $N$, we are able to fully unleash the high-order convergence potential of the RQMC method within a much lower-dimensional sampling space.

\section{Error and Dimensionality Reduction Effect of the Multilevel Stochastic Time Grid Method}\label{sec: qmc exact}

In the preceding section, we proved that directly applying the RQMC method can effectively improve the convergence rate of the sampling error in SDE simulations, thereby outperforming the classical MC method. However, when examining the overall simulation error, one must simultaneously consider the sampling error (determined by the sample size $m$) and the discretization error (determined by the number of time steps $N$). The classical EM scheme exhibits first-order weak convergence, meaning its discretization bias is approximately $\mathcal{O}(N^{-1})$. This implies that to achieve a prescribed overall error tolerance $\varepsilon$, the required number of time grid points $N$ must reach at least the order of $\mathcal{O}(\varepsilon^{-1})$.

This linear dependence on $N$ leads to a prominent issue: the integration dimension required to construct the RQMC estimator is $d \times N$. When the accuracy requirement is stringent (i.e., $\varepsilon$ is small), $N$ increases rapidly, rendering the actual integration dimension exceedingly high. Although modern RQMC theory and practice indicate that the high-dimensional challenge is not insurmountable—meaning that theoretical RQMC error bounds are often overly conservative and RQMC generally still outperforms MC in practical high-dimensional problems—it is undeniable that as the dimension escalates, it degrades the uniformity of point sets in high dimensions, thereby leading to diminishing marginal returns in algorithmic efficiency.

To mitigate the efficiency degradation caused by high dimensions, existing literature frequently employs path generation techniques such as Brownian Bridge or Principal Component Analysis. The core idea behind these methods is to rearrange the generation sequence of the Brownian motion increments, thereby reducing the \textbf{effective dimension} of the problem and concentrating the principal fluctuations of the integrand into the first few dimensions. Nevertheless, these methods do not alter the \textbf{actual integration dimension} of the problem; that is, the total number of underlying random variables remains $d \times N$.

Departing from the aforementioned approaches, in this section we introduce a novel perspective based on Exact Simulation and the Stochastic Time Grid, aiming to directly reduce the \textbf{actual integration dimension} of the problem. We observe that the decay of the discretization error of the classical EM scheme with respect to the number of grid points $N$ exhibits a "heavy-tailed" characteristic (i.e., a slow polynomial decay of $\mathcal{O}(N^{-1})$). This entails a prohibitive dimensional cost to eliminate the bias. Conversely, the discretization error based on the Exact Simulation method exhibits a ``thin-tailed", remarkably fast super-exponential decay as the number of discrete grid points (corresponding to the dimension) increases.

This ``thin-tailed" characteristic implies that we only need to simulate extremely low-dimensional variables to capture the essential features of the target function. Consequently, while ensuring the desired accuracy, the actual integration dimension required by the RQMC method can be constrained to a minimal range. In this manner, for a given error tolerance, we fundamentally reduce the actual integration dimension compared to the classical EM discretization scheme, significantly enhancing the computational efficiency of the RQMC method in SDE simulations.

In this section, we consider the case of constant volatility, where $\sigma(t,x) = \sigma_0$ is a invertible constant matrix. Thus, the process $X_t$ is governed by:
\begin{equation}
    \dd X_t = b(t,X_t) \dd t + \sigma_0 \dd W_t,\quad X_0 = x_0.
\end{equation}

\subsection{Exact Simulation Method}

\cite{henry-labordere2017Unbiased} proposed an unbiased estimator for the expectation $\E\left[ g(X_T)\right]$, where $g: \R^d \to \R$ is a given terminal function.
Let $\beta > 0$ be a fixed constant, and let $(\tau_j)_{j\ge 1}$ be a sequence of independent and identically distributed random variables following an exponential distribution $\mathcal{E}(\beta)$ with parameter $\beta$.
We define the partial sums $S_i = \sum_{j=1}^i \tau_j$ (with the convention $S_0=0$) and let $N_T = \max\{i: S_i < T\}$ denote the number of jumps prior to time $T$.
Then, $(N_t)_{0\le t\le T}$ constitutes a Poisson process with intensity $\beta$.
Utilizing the jump times of this process, we can construct a Stochastic Time Grid $\mathcal{T} = \{0 = T_0 < T_1 < \cdots < T_{N_T} < T_{N_T+1} = T\}$ on the interval $[0, T]$, where $T_i = S_i$ for $1 \le i \le N_T$.

Let $\Delta T_{i+1} := T_{i+1} - T_{i}$ denote the time intervals, and let $\Delta W_{T_{i+1}} := W_{T_{i+1}} - W_{T_i}$ represent the corresponding Brownian motion increments. We define the following discretized process $\widehat{X}$:
\begin{equation}
    \widehat{X}_{T_{i+1}} = \widehat{X}_{T_i} + b(T_i, \widehat{X}_{T_i})\Delta T_{i+1} + \sigma_0\Delta W_{T_{i+1}}, \quad \widehat{X}_{T_0} = x_0, \quad 0 \le i \le N_T.
\end{equation}
It is worth noting that although the construction of $\widehat{X}$ mimics the classical EM scheme, its time steps $\Delta T_{i+1}$ are fundamentally random.
Based on the above process, the estimator $\widehat{\psi}$ is defined as:
\begin{equation}\label{eq: unbiased estimator formula}
    \widehat{\psi} := e^{\beta T}\left[ g(\widehat{X}_T) - g(\widehat{X}_{T_{N_T}})\1_{\left\{ N_T > 0\right\}}\right] \prod_{i=1}^{N_T} \mathcal{W}_i, 
\end{equation}
where the weight term $\mathcal{W}_i$ is given by:
\begin{equation}\label{eq: weight definition}
    \mathcal{W}_i := \frac{\left(b(T_{i}, \widehat{X}_{T_i}) - b(T_{i-1}, \widehat{X}_{T_{i-1}})\right)^{\top} (\sigma_0^{\top})^{-1}\Delta W_{T_{i+1}}}{\beta \Delta T_{i+1}}.
\end{equation}

We review the main conclusion from \cite{henry-labordere2017Unbiased} in the following lemma.

\begin{assumption}\label{assum: for exact}
    Assume that $\sigma_0$ is positive definite, and the drift coefficient $b(t,x)$ is bounded and continuous with respect to $(t,x)$. Moreover, $b$ is uniformly $1/2$-H\"older continuous in $t$ and Lipschitz continuous in $x$; that is, there exists a constant $C>0$ such that for any $(t,x),(s,y) \in [0,T]\times \R^d$,
    \begin{equation}
        \left| b(t,x) - b(s,y)\right| \le C\left( \left| t-s\right|^{1/2} + \left| x-y\right|\right).
    \end{equation}
    Furthermore, assume that the terminal function $g$ satisfies a global Lipschitz condition.
\end{assumption}

\begin{lemma}[\cite{henry-labordere2017Unbiased}]
    Under Assumption \ref{assum: for exact}, $\widehat{\psi}$ is an unbiased estimator of the target expectation and possesses a finite second moment; that is,
    \begin{equation}
        \E[\widehat{\psi}] = \E[g(X_T)], \quad \text{and} \quad \E[\widehat{\psi}^2] < \infty.
    \end{equation}
\end{lemma}

\begin{remark}
    The classical EM scheme incurs discretization bias because it "freezes" the drift coefficient $b(t, X_t) \approx b(t_i, X_{t_i})$ over the time interval $[t_i, t_{i+1}]$. In contrast, formula \eqref{eq: unbiased estimator formula} introduces a Poisson process to transform a series summation into an expectation over a random stopping time $N_T$, thereby completely eliminating the systematic bias induced by time discretization in a statistical sense. The weight term $\mathcal{W}_i$ can be interpreted as a Girsanov-type correction factor. On each stochastic interval $[T_i, T_{i+1}]$, the approximating process $\widehat{X}$ utilizes a constant drift. The difference term $b(T_{i}, \widehat{X}_{T_i}) - b(T_{i-1}, \widehat{X}_{T_{i-1}})$ within the weight captures the variation of the drift coefficient over time. By adjusting the probability measure of the paths, this weight effectively compensates for the error originating from the freezing of the drift coefficient.
    
\end{remark}

\subsection{The Multilevel Stochastic Time Grid Estimator}

In this subsection, we introduce the Multilevel Stochastic Time Grid (MSTG) estimator. Let $N_T$ be a Poisson random variable with intensity $\beta T$. Denoting $p_n := \mathbb{P}(N_T = n) = \frac{(\beta T)^n}{n!}e^{-\beta T}$ for $n \ge 0$, we have:
\begin{equation}
    \E[\widehat{\psi}] = \sum_{n = 0}^\infty \E\left[ \widehat{\psi}\1_{\left\{ N_T = n\right\}}\right] = \sum_{n = 0}^\infty p_n \E\left[ \widehat{\psi} \mid N_T = n\right].
\end{equation}
We now focus on the conditional expectation $\E\left[ \widehat{\psi} \mid N_T = n\right]$ and aim to construct its unbiased estimator. Under the condition $N_T = n$, there are $n$ arrivals before time $T$, which partition the time interval $[0, T]$ into $n+1$ segments with lengths $\Delta T_1, \dots, \Delta T_{n+1}$. The following lemma characterizes the conditional distribution of these increments. 

\begin{lemma}\label{lemma: conditional distribution of time diff}
    Conditioned on $N_T = n$, the vector $(\Delta T_1/T, \dots, \Delta T_{n+1}/T)$ follows a uniform distribution on the $n$-dimensional simplex: $$\mathcal{A}:=\left\{\bm y \in \R^{n+1}: \sum_{i=1}^{n+1} y_i = 1, y_i \ge 0, i = 1, \dots, n+1\right\}.$$
\end{lemma}
\begin{proof}
    Conditioned on $N_T=n$, the arrival times $(T_1,\dots,T_n)$ have the same distribution as the order statistics of $n$ independent uniform random variables on $[0,T]$ \cite[Section~2.5]{gallager1996discrete}. The normalized spacings of uniform order statistics have the Dirichlet distribution $\mathrm{Dirichlet}(1,\dots,1)$, which is the uniform distribution on the simplex \cite{pyke1965spacings,david2003orderstatistics}.
\end{proof}

Lemma \ref{lemma: conditional distribution of time diff} establishes the distribution of $(\Delta T_1, \dots, \Delta T_{n+1})$ given $N_T = n$. While various methods exist for simulating a uniform distribution on a simplex—such as using the spacings of $n$ order statistics—the latter involves indicator functions. Such non-smoothness introduces points of non-differentiability in the integrand, which can severely degrade the QMC convergence rate \cite{he2015}. The following lemma provides an alternative simulation approach using normalized exponential random variables (see more details in \cite[Chapter 5]{owen2013}), which ensures a smoother integrand by avoiding indicator functions. 

\begin{lemma}\label{lemma: normalized exponential}
    Let $\tau_1, \dots, \tau_{n+1}$ be independent and identically distributed exponential random variables. Then the normalized random vector $\left(\frac{\tau_1}{\sum_{i=1}^{n+1} \tau_i}, \dots, \frac{\tau_{n+1}}{\sum_{i=1}^{n+1} \tau_i} \right)$ follows a uniform distribution on the $n$-dimensional simplex $ \mathcal{A}$. 
\end{lemma}

Utilizing Lemma \ref{lemma: conditional distribution of time diff} and Lemma \ref{lemma: normalized exponential}, we construct an estimator to approximate $\E[g(X_T)]$. 

For a fixed level $n$, we characterize the estimator as a deterministic function of a Gaussian input vector $\bm{z} \in \mathbb{R}^{(d+1)(n+1)}$. We decompose $\bm{z} = (\bm z^{\tau},\bm z^W)$ into a time-grid component $\bm{z}^{\tau} = (z^{\tau}_1, \dots, z^{\tau}_{n+1}) \in \mathbb{R}^{n+1}$ and a Brownian component $\bm{z}^W = (\bm{z}^{W}_1, \dots, \bm{z}^{W}_{n+1})\in \mathbb{R}^{d(n+1)}$, where each $\bm{z}^{W}_i \in \mathbb{R}^d$. 

First, we define the stochastic time increments as functions of $\bm{z}^{\tau}$ through the composition of the inverse exponential CDF $ F_{\tau}^{-1}$ and the standard normal CDF $ \Phi$. The increments are given by:
\begin{equation}\label{eq: time increment}
    \Delta T_i(\bm{z}^{\tau}) := \frac{F_{\tau}^{-1} \circ \Phi(z_i^{\tau})}{\sum_{j = 1}^{n+1} F_{\tau}^{-1} \circ \Phi(z_j^{\tau})} T, \quad i = 1, \dots, n+1.
\end{equation}
The corresponding time grid points are then $T_i(\bm{z}^{\tau}) = \sum_{j=1}^i \Delta T_j(\bm{z}^{\tau})$ with $T_0 = 0$. This construction ensures that when $\bm z^\tau \sim \mathcal{N}(0,I_{n+1})$, the vector $(\Delta T_1, \dots, \Delta T_{n+1})$ is uniformly distributed on the $n$-dimensional simplex.

Second, the Brownian increments are constructed as functions of both the time component and the Gaussian Brownian component:
\begin{equation}\label{eq: brownian increment}
    \Delta W_{T_i}(\bm z) := \sqrt{\Delta T_i(\bm{z}^{\tau})} (\bm{z}^W_i)^{\top}, \quad i = 1, \dots, n+1.
\end{equation}
For a given input $\bm{z}$, the path of the discretized SDE is generated iteratively:
\begin{equation}\label{eq: state process}
    \widehat{X}_{T_{i+1}}(\bm{z}) = \widehat{X}_{T_i}(\bm{z}) + b(T_i(\bm{z}^{\tau}), \widehat{X}_{T_i}(\bm{z}))\Delta T_{i+1}(\bm{z}^{\tau}) + \sigma_0\Delta W_{T_{i+1}}(\bm z),
\end{equation}
with $\widehat{X}_{T_0} = x_0$ for $0 \le i \le n$.

Finally, the level-specific function $\widehat{\psi}_n(\bm{z})$ is defined as:
\begin{equation}\label{eq: level estimator}
    \widehat \psi_n(\bm{z}) := e^{\beta T}\left[ g(\widehat{X}_{T_{n+1}}(\bm{z})) - g(\widehat{X}_{T_n}(\bm{z}))\1_{\left\{ n > 0\right\}}\right] \prod_{i =1}^n \mathcal{W}_i(\bm{z}),
\end{equation}
with the convention $\prod_{i=1}^0 = 1$. The weighting functions are:
\begin{equation}
    \mathcal{W}_i(\bm{z}) := \frac{\left(b(T_{i}(\bm{z}^{\tau}),\widehat{X}_{T_i}(\bm{z})) - b(T_{i-1}(\bm{z}^{\tau}),\widehat{X}_{T_{i-1}}(\bm{z}))\right)^{\top} (\sigma_0^{\top})^{-1}\Delta W_{T_{i+1}}(\bm{z})}{\beta \Delta T_{i+1}(\bm{z}^{\tau})}.
\end{equation}
For the level $n$ and corresponding sample size $m_n := [mp_n]$, let $\{(\boldsymbol{y}^{(k)})\}_{k=1}^{m_n}$ be a low-discrepancy point set in the unit cube $[0,1]^{(d+1)(n+1)}$. We set $\bm{z}^{(k)} = \Phi^{-1}(\bm{y}^{(k)})$ (applied component-wise). The MSTG estimator for a total sample parameter $m$ is then:
\begin{equation}\label{eq: mstg}
   \hat I_{m,N}^{\textrm{RQMC}} := \frac{1}{m}\sum_{n = 0}^N \sum_{k=1}^{m_n} \widehat\psi_n(\bm{z}^{(k)}).
\end{equation}

\subsection{Error Analysis}

This section is devoted to the error analysis of the MSTG estimator \eqref{eq: mstg}. To avoid overly cumbersome notation, our analysis focuses on the benchmark case where $T=1$ and $d=1$. The conclusions can be naturally extended to higher-dimensional spaces. Furthermore, we assume that the drift coefficient $b$ is time-homogeneous. This simplification is intended to streamline the technical details and highlight the main logic of the proof; the analysis for the non-homogeneous case follows an entirely analogous procedure. To this end, we impose the following smoothness conditions on the drift coefficient:

\begin{assumption}\label{assum: smooth drift}
    The drift coefficient is time-homogeneous, i.e., $b(t, x) = b(x)$. We assume $b$ and $ g$ are smooth with bounded derivatives of all orders; that is,  for any $k \ge 1$, $\frac{\dd^k}{\dd x^k} b(x)$, $ \frac{\dd^k}{\dd x^k} g(x)$ are continuous and there exists a constant $C_k > 0$ such that $\left| \frac{\dd^k}{\dd x^k} b(x) \right| + \left| \frac{\dd^k}{\dd x^k} g(x)\right| \le C_k$.
\end{assumption}

The RMSE between the MSTG estimator \eqref{eq: mstg} and the true value $\E[\widehat{\psi}] = \E[g(X_T)]$ can be decomposed as follows:
\begin{equation}\label{eq: mstg error decom}
    \begin{aligned}
        &\sqrt{\E\left[ \left| \hat I_{m,N}^{\textrm{RQMC}} - \E[g(X_T)]\right|^2\right]} \\
        & = \sqrt{\E\left[\left| \hat I_{m,N}^{\textrm{RQMC}} - \sum_{n=0}^\infty p_n \E[\widehat \psi\mid N_T = n]\right|^2\right]}\\
        &\le \underbrace{\sqrt{\E\left[\left| \hat I_{m,N}^{\textrm{RQMC}} - \sum_{n=0}^N p_n \E[\widehat \psi\mid N_T = n]\right|^2\right]}}_{\text{Sampling Error}} + \underbrace{\sum_{n = N+1}^{\infty}p_n \left| \E[\widehat \psi\mid N_T = n]\right|}_{\text{Truncation Error}}
    \end{aligned}
\end{equation}

In the following, we estimate the \textbf{truncation error} and the \textbf{sampling error} separately. We begin with a lemma providing an upper bound for the truncation error.

\begin{lemma}\label{lemma: mstg truncation error}
    Under Assumption \ref{assum: for exact}, there exists a constant $C$ such that:
    \begin{equation}
        \sum_{n = N+1}^{\infty}p_n \left| \E[\widehat \psi\mid N_T = n]\right| \le CT^{1/2}(N/(eT))^{-N}.
    \end{equation}
\end{lemma}
\begin{proof}
    See Section \ref{sec proof lemma: mstg truncation error}.
\end{proof}

Next, we derive the upper bound for the sampling error. Note that the sampling error admits the following decomposition:
\begin{equation}\label{eq: sample error bound 1}
    \begin{aligned}
        &\sqrt{\E\left[\left| \hat I_{m,N}^{\textrm{RQMC}} - \sum_{n=0}^N p_n \E[\widehat \psi\mid N_T = n]\right|^2\right]} \\
        &= \sqrt{\E\left[ \left|\frac{1}{m}\sum_{n=0}^{N}\sum_{k=1}^{m_n}\widehat\psi_n(\bm{z}^{(k)}) - \sum_{n=0}^Np_n\E[\widehat \psi\mid N_T = n]\right|^2\right]}\\
         & = \sqrt{\E\left[ \left|\sum_{n=0}^N\frac{m_n}{m} \frac{1}{m_n}\sum_{k=1}^{m_n}\widehat\psi_n(\bm{z}^{(k)}) - \sum_{n=0}^Np_n\E[\widehat \psi\mid N_T = n]\right|^2\right]}\\
         &\le \sum_{n=0}^N \frac{m_n}{m}\sqrt{\E\left[ \left| \frac{1}{m_n} \sum_{k=1}^{m_n}\widehat\psi_n(\bm{z}^{(k)}) - \E[\widehat \psi\mid N_T = n]\right|^2\right]} \\
         &\qquad+ \sum_{n=0}^N \sqrt{\E\left[ \left| \frac{m_n}{m} \E[\widehat \psi\mid N_T = n] - p_n \E[\widehat \psi\mid N_T = n]\right|^2\right]}
    \end{aligned}
\end{equation}
For the second term on the right-hand side, we have the following bound:
\begin{equation}\label{eq: sample error bound 2}
    \begin{aligned}
        &\sum_{n=0}^N \sqrt{\E\left[ \left| \frac{m_n}{m} \E[\widehat \psi\mid N_T = n] - p_n \E[\widehat \psi\mid N_T = n]\right|^2\right]} \\
        &= \sum_{n=0}^N\left| \frac{[mp_n]}{m} - \frac{mp_n}{m}\right| \left| \E[\widehat \psi\mid N_T = n]\right|\\  
        &\le \frac{1}{m}\sum_{n=0}^N \left| \E[\widehat \psi\mid N_T = n]\right|\le \frac{1}{m}\sum_{n=0}^N CT^{1/2}e^{\beta T} \beta^{-n}\\
        &\le C_N\frac{1}{m},
    \end{aligned}
\end{equation}
where for the first inequality, we used the fact that $|[mp_n]/m - p_n| \le 1/m$, for the second inequality, we used the upper bound of $\E[|\widehat{\psi}_n|]$ from \eqref{eq: upper bound of psi_n}.

Consequently, it remains to bound the RMSE of the sample mean at each level:
\begin{equation}\label{eq: level qmc error}
    \sqrt{\E\left[ \left| \frac{1}{m_n} \sum_{k=1}^{m_n}\widehat\psi_n(\bm{z}^{(k)}) -   \E[\widehat \psi\mid N_T = n]\right|^2\right]}.
\end{equation}
This is a standard RQMC error bound for a sample size $m_n$. 

Recall that the level-specific estimator has been formulated as a deterministic function of the Gaussian input vector, $\widehat{\psi}_n(\bm{z})$.  According to the RQMC theory for unbounded integrands established in \cite[Corollary 4.8]{ouyang2024achieving}, achieving a sampling error convergence rate of $\mathcal{O}(m_n^{-1+\eps})$ requires the integrand to satisfy a specific regularity condition. Therefore, it suffices to verify that the function $\widehat{\psi}_n(\bm{z})$ exhibits polynomial growth.

The remainder of this section is devoted to rigorously proving this property. We proceed systematically by establishing the growth bounds for each functional component: first for the stochastic time increments $\Delta T_i(\bm{z}^{\tau})$, then for the discrete state variables $\widehat{X}_{T_i}(\bm{z})$, and ultimately for the estimator $\widehat{\psi}_n(\bm{z})$ itself. Specifically, we will demonstrate that $\widehat{X}_{T_i}$, $\widehat{\psi}_n$, and all their associated mixed partial derivatives strictly belong to the polynomial growth function class $\mathcal{C}_p(\R^{2n+2})$.

\paragraph{Growth of the Stochastic Time Step}
We first examine the partial derivatives of the stochastic time step $\Delta T_i$ with respect to the underlying normal variables $z_j^\tau$.
\begin{lemma}\label{lemma: deri delta t growth}
    There exists a constant $C > 0$ such that for any $\boldsymbol u^{\tau} \subseteq 1{:}(n+1)$ and $1 \le i \le n+1$:
    \begin{equation}
        \left| \partial^{\boldsymbol u^{\tau}}_{\boldsymbol z^{\tau}}\Delta T_i\right| \le C\Delta T_i \prod_{j \in \boldsymbol{u}^\tau}|z_j^{\tau}|.
    \end{equation}
\end{lemma}
\begin{proof}
    See Section \ref{sec proof lemma: deri delta t growth}.
\end{proof}

\paragraph{Polynomial Growth of State Variables}
We now consider the growth of the discrete state variables $\widehat{X}_{T_i}$ constructed on the stochastic time grid.
\begin{lemma} \label{lemma: deri x growth}
Under Assumptions \ref{assum: for exact} and \ref{assum: smooth drift}, for any $0 \le k \le n+1$, $\widehat{X}_{T_k} \in \mathcal{C}_p(\R^{2n+2})$. Furthermore, for all $\bm u^{\tau}, \bm u^W \subseteq 1{:}(n+1)$, the mixed partial derivatives satisfy $\partial^{\bm u^{\tau}}_{\bm z^{\tau}}\partial^{\bm u^W}_{\bm z^W} \widehat{X}_{T_k} \in \mathcal{C}_p(\R^{2n+2})$.
\end{lemma}
\begin{proof}
    See Section \ref{sec proof lemma: deri x growth}.
\end{proof}

\paragraph{Polynomial Growth of the function $\widehat{\psi}_n$}
For $n \ge 1$, we recall the expression for the stochastic time grid estimator $\widehat \psi_n$:
\begin{equation}\label{eq: psi expression}
    \begin{aligned}
    \widehat \psi_n 
    &= e^{\beta T}\left( g(\widehat X_{T_{n+1}}) - g(\widehat X_{T_n})\right) \prod_{i = 1}^n \frac{\left( b(\widehat{X}_{T_i}) - b(\widehat{X}_{T_{i-1}})\right)(\sigma_0)^{-1}\Delta W_{T_{i+1}}}{\beta \Delta T_{i+1}}\\
    & = (\beta\sigma_0)^{-n}e^{\beta T}\big(\prod_{i=1}^n z^{W}_{i+1}\big)\left( b(\widehat{X}_{T_1}) - b(\widehat{X}_{T_0})\right)\frac{ g(\widehat X_{T_{n+1}}) - g(\widehat X_{T_n})}{\sqrt{\Delta T_{n+1}}}\prod_{i = 2}^n\frac{b(\widehat{X}_{T_i}) - b(\widehat{X}_{T_{i-1}})}{\sqrt{\Delta T_i}}.
    \end{aligned}
\end{equation}
From this expression, it is evident that we only need to prove that the derivatives of terms such as $\frac{f(\widehat{X}_{T_i}) - f(\widehat{X}_{T_{i-1}})}{\sqrt{\Delta T_i}}$ belong to the class $\mathcal{C}_p(\R^{2n+2})$ (Here, $ f = g$ or $ b$).
\begin{lemma}\label{lemma: deri diff growth}
    Under the conditions in Assumption \ref{assum: smooth drift}, let $ f = b$ or $ g$, then for any $1 \le i \le n+1$ and $\bm u^{\tau}, \bm u^{W} \subseteq 1{:}(n+1)$, we have
    \begin{equation}\label{eq: diff target}
        \partial^{\bm u^{\tau}}_{\bm z^{\tau}}\partial^{\bm u^W}_{\bm z^W}\frac{f(\widehat{X}_{T_i}) - f(\widehat{X}_{T_{i-1}})}{\sqrt{\Delta T_i}}  \in \mathcal{C}_p(\R^{2n+2}).
    \end{equation}
\end{lemma}
\begin{proof}
    See Section \ref{sec proof lemma: deri diff growth}.
\end{proof}

Combining Lemmas \ref{lemma: deri delta t growth}, \ref{lemma: deri x growth}, and \ref{lemma: deri diff growth} with the expression for $\widehat{\psi}_n$ \eqref{eq: psi expression}, we obtain the following theorem:
\begin{theorem}\label{them: mstg qmc error}
    Under Assumptions \ref{assum: for exact} and \ref{assum: smooth drift}, there exist constants $M,B,K > 0$ such that:
    \begin{equation}
        \widehat\psi_n \in G_p(M,B,K).
    \end{equation}
\end{theorem}

If $\{\bm y^{(k)} = (\bm y^{\tau,(k)}, \bm y^{W,(k)})\}_{k=1}^{m_n}$ are the first $m_n$ points of a Owen-scrambled Sobol' points, they satisfy the discrepancy bound $D^{*}_{m_n} \le C_n \frac{(\log m_n)^{2(n+1)}}{m_n}$ with probability 1. Based on \cite[Corollary 4.8]{ouyang2024achieving}, we arrive at the following convergence result.

\begin{theorem}\label{them: mstg converge}
    Under Assumptions \ref{assum: for exact} and \ref{assum: smooth drift}, if we use RQMC point set for each level $n$, then the RMSE of the level-specific estimator is bounded by:
    \begin{equation}\label{eq: mstg level error}
        \sqrt{\E\left[ \left| \frac{1}{m_n} \sum_{k=1}^{m_n}\widehat\psi_n(\bm{z}^{(k)}) - \E[\widehat\psi \mid N_T = n]\right|^2\right]} \le C_n \frac{(\log m_n)^{3(n+1) + K_n}}{m_n}.
    \end{equation}
    Furthermore, if $m_n = [mp_n]$, then  there are constants $C(\eps,N), C > 0$ such that for any $ \eps > 0$,
    \begin{equation}\label{eq: final error bound}
        \sqrt{\E\left[ \left| \hat I_{m,N}^{RQMC} - \E[g(X_T)]\right|^2\right]}  \le C(\eps,N)m^{-1+\eps} + C(N/(eT))^{-N}.
    \end{equation}
\end{theorem}
\begin{proof}
According to Theorem \ref{them: mstg qmc error}, for each level $n$, there exist constants $M_n, B_n, K_n > 0$, such that
\begin{equation}
\widehat\psi_n \in G_p(M_n,B_n,K_n).
\end{equation}
Equation \eqref{eq: mstg level error} then follows by applying \cite[Corollary 4.8]{ouyang2024achieving}. Next, we prove equation \eqref{eq: final error bound}. Recalling the error analysis of the estimator \eqref{eq: mstg error decom}, and using Lemma \ref{lemma: mstg truncation error} to bound the truncation error, it suffices to prove the convergence order of the sampling error. Recalling the previous decomposition and upper bound estimates for the sampling error in equations \eqref{eq: sample error bound 1} and \eqref{eq: sample error bound 2}, and again applying Theorem \ref{them: mstg qmc error}, we have
\begin{equation}
        \begin{aligned}
            &\sqrt{\E\left[\left| \hat I_{m,N}^{\textrm{RQMC}} - \sum_{n=0}^N p_n \E[\widehat \psi\mid N_T = n]\right|^2\right]}\\
         &\le \sum_{n=0}^N \frac{m_n}{m}\sqrt{\E\left[ \left| \frac{1}{m_n} \sum_{k=1}^{m_n}\widehat\psi_n(\bm{z}^{(k)}) - \E[\widehat \psi\mid N_T = n]\right|^2\right]} + C_N\frac{1}{m}\\
         &\le \sum_{n=0}^N \frac{m_n}{m} C_n\frac{(\log m_n)^{3(n+1) + K_n}}{m_n} + C_N\frac{1}{m}\\
         &  =   \frac{\sum_{n = 0}^N C_n(\log m_n)^{3(n+1) + K_n}}{m} + C_N\frac{1}{m}.
        \end{aligned}
    \end{equation}
This completes the proof.
\end{proof}

Comparing Theorem \ref{thm: qmc error} with Theorem \ref{them: mstg converge} clarifies the mechanism by which the MSTG construction improves the efficiency of RQMC sampling for SDE simulation.

\paragraph{Discretization and truncation errors}
Theorem \ref{thm: qmc error} shows that the bias of the classical EM estimator is of order $\mathcal{O}(N_{\mathrm{EM}}^{-1})$, where $N_{\mathrm{EM}}$ denotes the number of EM time steps. Hence, reducing the bias to a prescribed level $\delta$ requires $N_{\mathrm{EM}}=\mathcal{O}(\delta^{-1})$ time steps, and therefore an input dimension proportional to $dN_{\mathrm{EM}}$. By contrast, the truncation error of the MSTG estimator, given by the second term in \eqref{eq: final error bound}, is of order $\mathcal{O}((N_{\mathrm{MSTG}}/(eT))^{-N_{\mathrm{MSTG}}})$, where $N_{\mathrm{MSTG}}$ is the maximal truncation level. This super-exponential decay implies that the truncation level required for the same tolerance grows only mildly with $\delta^{-1}$.

\begin{figure}[htbp]
    \centering
    \includegraphics[width=0.8\textwidth]{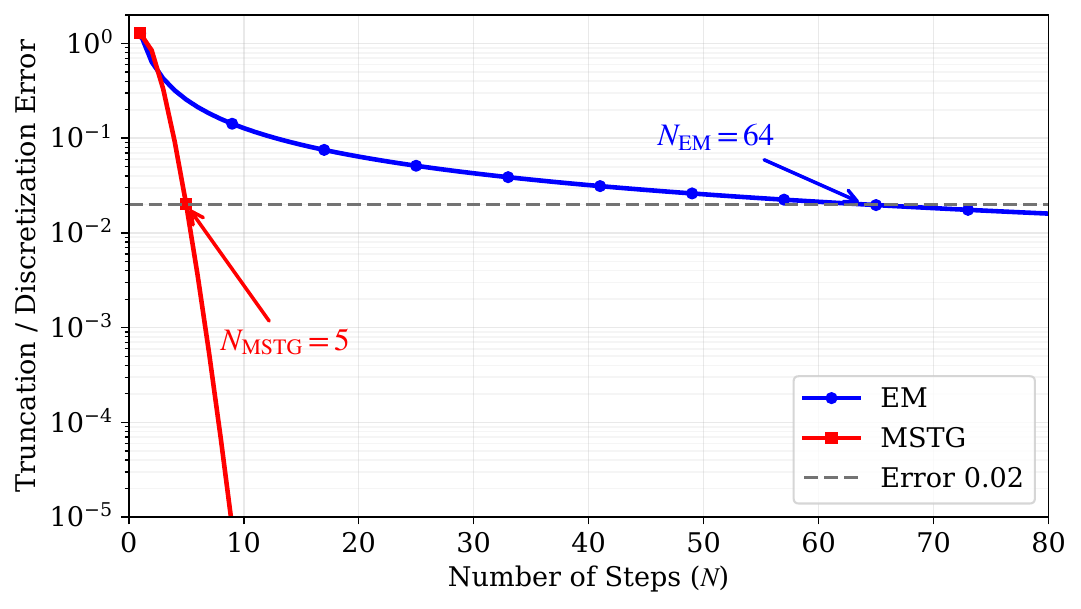} 
    \caption{Comparison of the tail behavior of discretization/truncation errors at $\beta=1$. The figure contrasts the polynomial decay of the EM scheme (blue circles) with the super-exponential decay of the exact simulation method (red squares) as a function of the dimension/steps $N$ on a semi-log scale.}
    \label{fig:tail_comparison}
\end{figure}

Figure \ref{fig:tail_comparison} illustrates this difference. The EM bias decreases according to a slow algebraic rate, whereas the MSTG truncation error decreases at a super-exponential rate. For example, for the tolerance level displayed in the figure, the EM discretization requires many more time steps than the MSTG truncation level. Since the input dimension of the EM-QMC estimator is proportional to the number of time steps, this difference directly affects the dimension in which the RQMC rule must operate.

\paragraph{Dimension reduction for RQMC}
The practical performance of RQMC depends not only on smoothness but also on how the integrand depends on the input coordinates. For the EM scheme, decreasing the bias forces the RQMC point set to be used in dimension $dN_{\mathrm{EM}}$. In the MSTG estimator, the contribution at level $n$ is integrated in dimension $(d+1)(n+1)$, and the super-exponential tail bound permits a small maximal level $N_{\mathrm{MSTG}}$. Thus the original high-dimensional fixed-grid integration problem is replaced by a finite weighted sum of lower-dimensional integrations. This is the dimension-reduction mechanism exploited below.

\section{Numerical Experiments}
\label{sec:numerical_experiments}

This section reports numerical experiments for the MSTG estimator combined with RQMC sampling. We compare the root mean square errors (RMSE) of four estimators:
\begin{itemize}
    \item \textbf{EM-MC}: the Euler-Maruyama estimator with standard Monte Carlo sampling;
    \item \textbf{EM-QMC}: the Euler-Maruyama estimator with Owen-scrambled Sobol' points;
    \item \textbf{MSTG-MC}: the multilevel stochastic time-grid estimator with standard Monte Carlo sampling;
    \item \textbf{MSTG-QMC}: the multilevel stochastic time-grid estimator with Owen-scrambled Sobol' points.
\end{itemize}

\subsection{One-dimensional SDE}
We first consider the benchmark example from \cite{henry-labordere2017Unbiased}. The target is the European-type payoff $\E[g(X_T)]$ with $g(x)=(e^x-1)^+$, where $X$ solves
\begin{equation}\label{eq: numerical sde}
    \dd X_t = \left( 0.1\sqrt{M_b \land e^{X_t}} - \frac{1}{8}\right)\dd t + \frac{1}{2} \dd W_t, \quad X_0 = 0,
\end{equation}
with $M_b=100$, $T=1$, and Poisson intensity $\beta=1$. Since no closed-form reference value is available, we use the high-accuracy MLMC value reported in \cite{henry-labordere2017Unbiased}, namely $\E[g(X_T)]\approx 0.205638$, as the benchmark. For all RQMC estimators we use Owen-scrambled Sobol' points, and RMSEs are estimated from 16 independent randomizations.

We examine two regimes. First, we set $N_{\mathrm{EM}}=10$ and $N_{\mathrm{MSTG}}=4$. This choice gives the same actual integration dimension for the two RQMC estimators, since the EM dimension is $1\times 10=10$, while the largest MSTG level has dimension $(1+1)(4+1)=10$. The results are shown in Figure \ref{fig:conv_low_dim}.

\begin{figure}[htbp]
    \centering
    \includegraphics[width=0.7\textwidth]{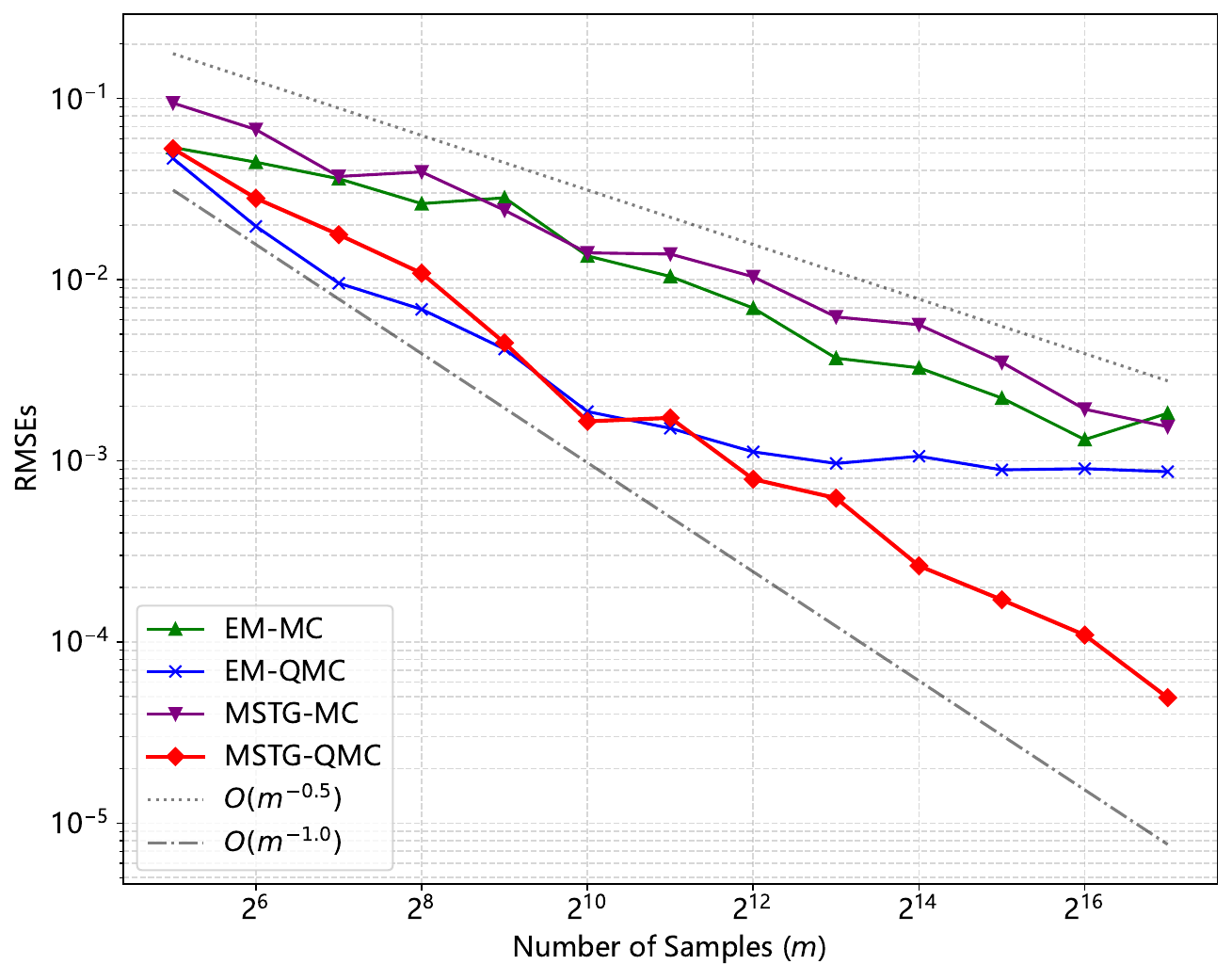} 
    \caption{One-dimensional example: RMSE versus sample size for $N_{\mathrm{EM}}=10$ and $N_{\mathrm{MSTG}}=4$, both with actual integration dimension $10$.}
    \label{fig:conv_low_dim}
\end{figure}

Figure \ref{fig:conv_low_dim} shows that, for both EM and MSTG, replacing MC by scrambled Sobol' sampling substantially improves the convergence rate. The MC curves are consistent with the standard $m^{-1/2}$ rate, whereas the RQMC curves are close to first order in $m$. The difference between EM-QMC and MSTG-QMC becomes more pronounced for larger sample sizes. This is consistent with the fact that, once the sampling error is reduced below the EM bias, the EM curves flatten, while the MSTG estimator continues to decrease because its truncation bias at this level is much smaller.

Second, we compare the methods at a comparable bias level. We increase the number of EM time steps to $N_{\mathrm{EM}}=64$, while the MSTG estimator uses only $N_{\mathrm{MSTG}}=5$. The latter choice is sufficient because of the super-exponential truncation bound in Theorem \ref{them: mstg converge}. The results are shown in Figure \ref{fig:conv_high_dim}.

\begin{figure}[htbp]
    \centering
    \includegraphics[width=0.7\textwidth]{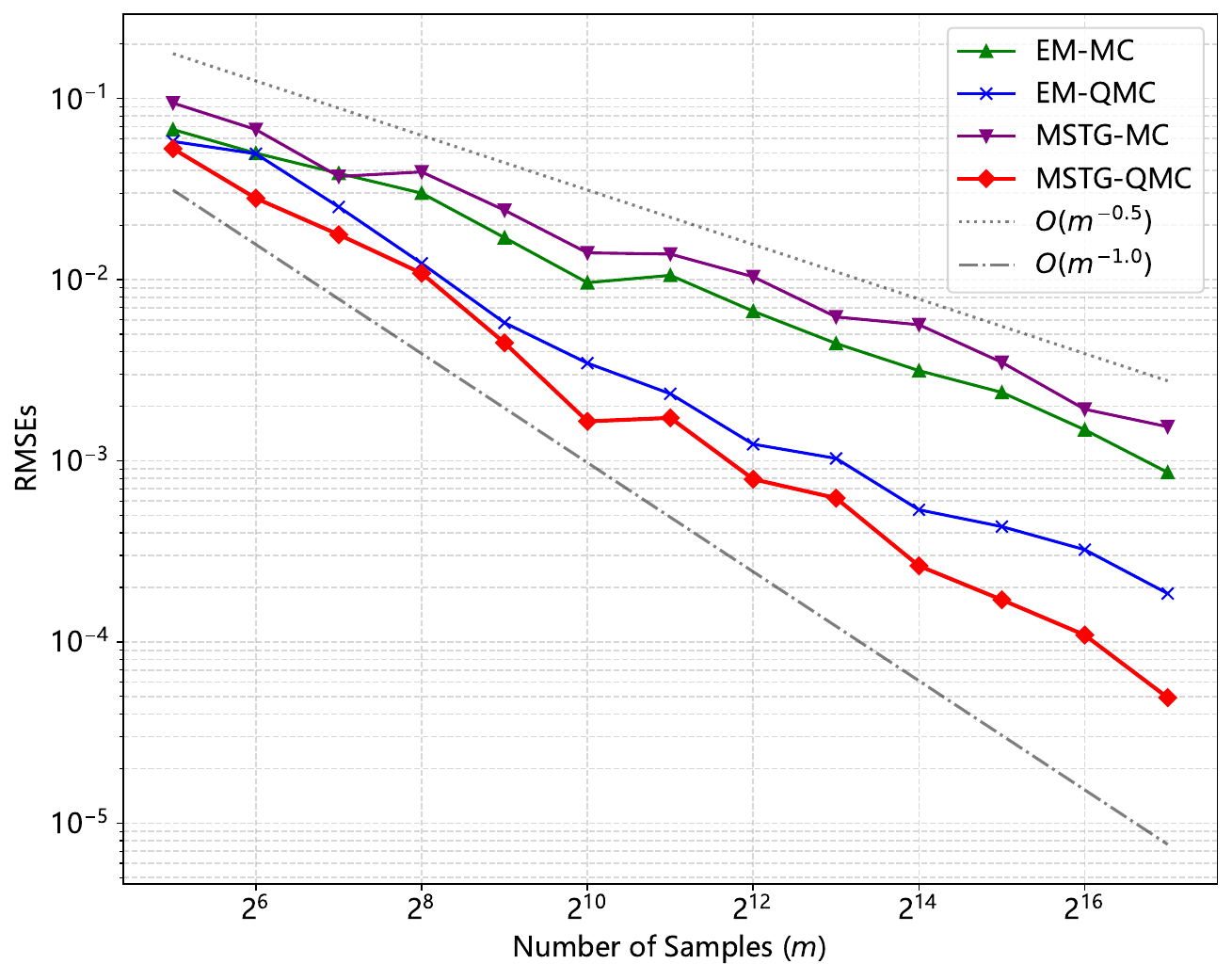} 
    \caption{One-dimensional example: RMSE versus sample size for $N_{\mathrm{MSTG}}=5$ and $N_{\mathrm{EM}}=64$.}
    \label{fig:conv_high_dim}
\end{figure}

The second comparison isolates the dimension-reduction effect. To reach a similar bias level, the EM estimator uses a 64-dimensional Gaussian input in this one-dimensional SDE, whereas MSTG-QMC only integrates over levels up to dimension $2(N_{\mathrm{MSTG}}+1)=12$. The lower-dimensional structure of the MSTG estimator leads to smaller errors and a more stable RQMC convergence curve. This behavior is precisely what the theory predicts: the sampling rate remains close to first order, while the truncation error is already negligible at a small maximal level.

\subsection{Multidimensional SDEs}
\label{subsec:numerical_multidim}

We next consider a family of $d$-dimensional nonlinear correlated basket-type SDEs, following the example in \cite{henry-labordere2017Unbiased}.  Let $W=(W^1,\ldots,W^d)^\top$ be a $d$-dimensional Brownian motion, and let $\sigma_0$ be the lower triangular Cholesky factor of the equicorrelation matrix
\[
C_d=(c_{ij})_{i,j=1}^d,\qquad
c_{ii}=1,\quad c_{ij}=\frac12\quad (i\ne j).
\]
The state process $X=(X^1,\ldots,X^d)^\top$ is defined by
\begin{equation}\label{eq: numerical sde multidim}
\dd X_t = \mu(X_t)\dd t + \sigma_0 \dd W_t,
\qquad X_0=0\in\mathbb{R}^d,
\end{equation}
where $\mu(x)=(\mu_i(x))_{i=1}^d$ is given by
\begin{equation}\label{eq: numerical drift multidim}
\mu_i(x)
=
0.1\left(\sqrt{M_b \wedge \left(\frac{3}{4}\exp(x_i)
+\frac{1}{4}\overline{\exp}_d(x)\right)}-1\right)
-\frac{1}{8},
\qquad
\overline{\exp}_d(x):=\frac1d\sum_{j=1}^d e^{x_j}.
\end{equation}
The quantity of interest is the basket payoff
\begin{equation}\label{eq: numerical payoff multidim}
V_0^{(d)}
:=
\E\left[
\left(
\frac1d\sum_{i=1}^d e^{M_b\wedge X_T^i}-K
\right)_+
\right].
\end{equation}
Throughout this multidimensional experiment we take $M_b=100$, $K=1$, $T=1$, and $\beta=1$.

\paragraph{The case $d=4$.}
We first set $d=4$, which gives the multidimensional benchmark in \cite{henry-labordere2017Unbiased}.  The reference value is obtained from a high-accuracy MLMC computation; we use $V_0^{(4)}=0.7030524641$.

The results for the four-dimensional problem are reported in Figures \ref{fig:conv_low_dim_4d} and \ref{fig:conv_high_dim_4d}. The same qualitative conclusions remain valid in this higher-dimensional setting. In the first comparison, we choose $N_{\mathrm{EM}}=10$ and $N_{\mathrm{MSTG}}=7$, so that the actual integration dimensions are identical: $4\times 10=(4+1)(7+1)=40$. With this dimension-matched setting, RQMC sampling improves both EM and MSTG over their MC counterparts, and MSTG-QMC continues to decrease after the EM-QMC curve is affected by discretization bias. When the EM step number is increased to $N_{\mathrm{EM}}=64$ while MSTG uses $N_{\mathrm{MSTG}}=5$, the contrast becomes stronger: the EM-QMC estimator is applied in dimension $4\times64=256$, whereas the largest MSTG level uses dimension $(4+1)(5+1)=30$. The numerical results show that this reduction in actual input dimension is important for maintaining the effectiveness of the scrambled Sobol' rule.

\begin{figure}[htbp]
    \centering
    \includegraphics[width=0.7\textwidth]{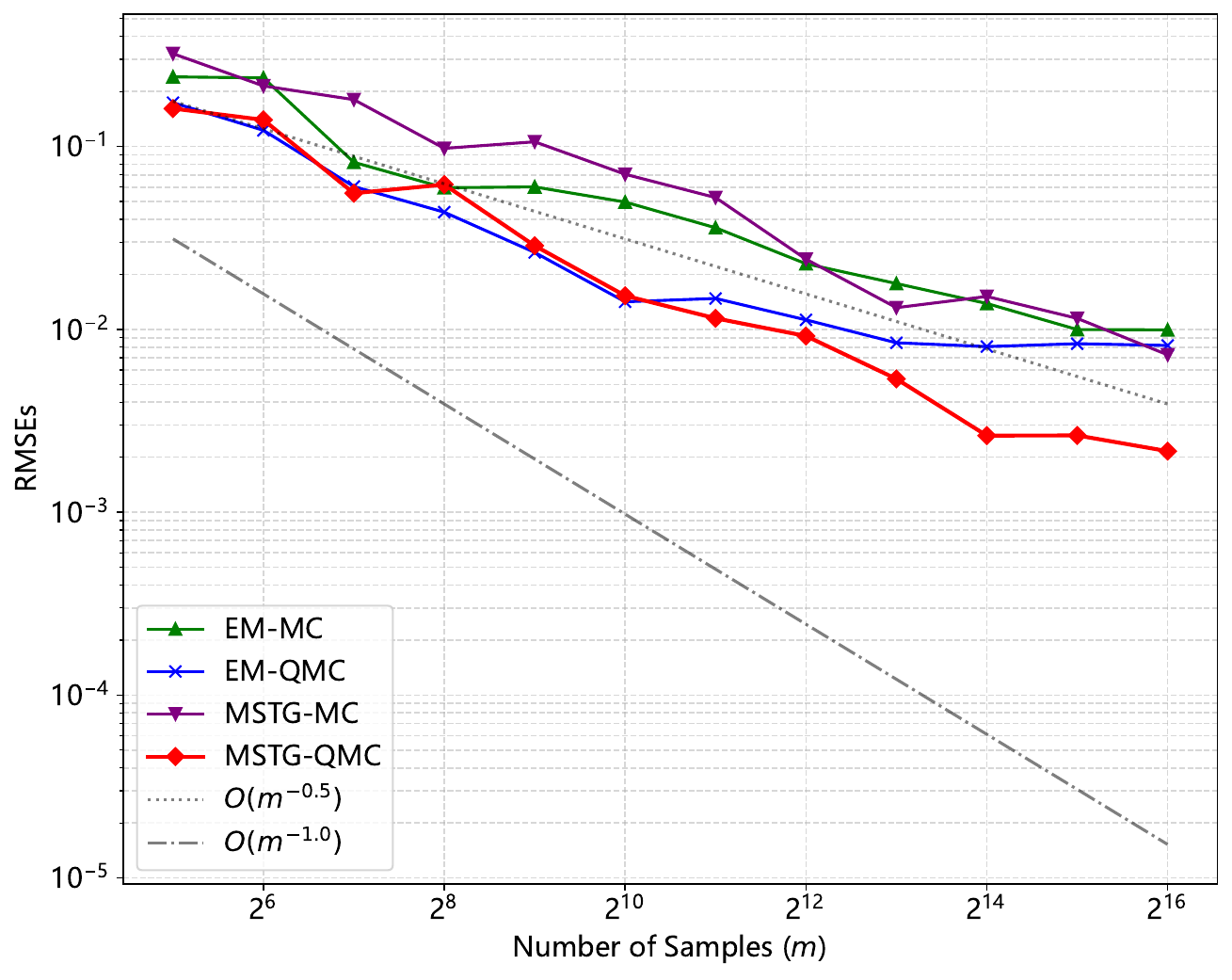}
    \caption{Case $d=4$: RMSE versus sample size for $N_{\mathrm{EM}}=10$ and $N_{\mathrm{MSTG}}=7$, both with actual integration dimension $40$.}
    \label{fig:conv_low_dim_4d}
\end{figure}

\begin{figure}[htbp]
    \centering
    \includegraphics[width=0.7\textwidth]{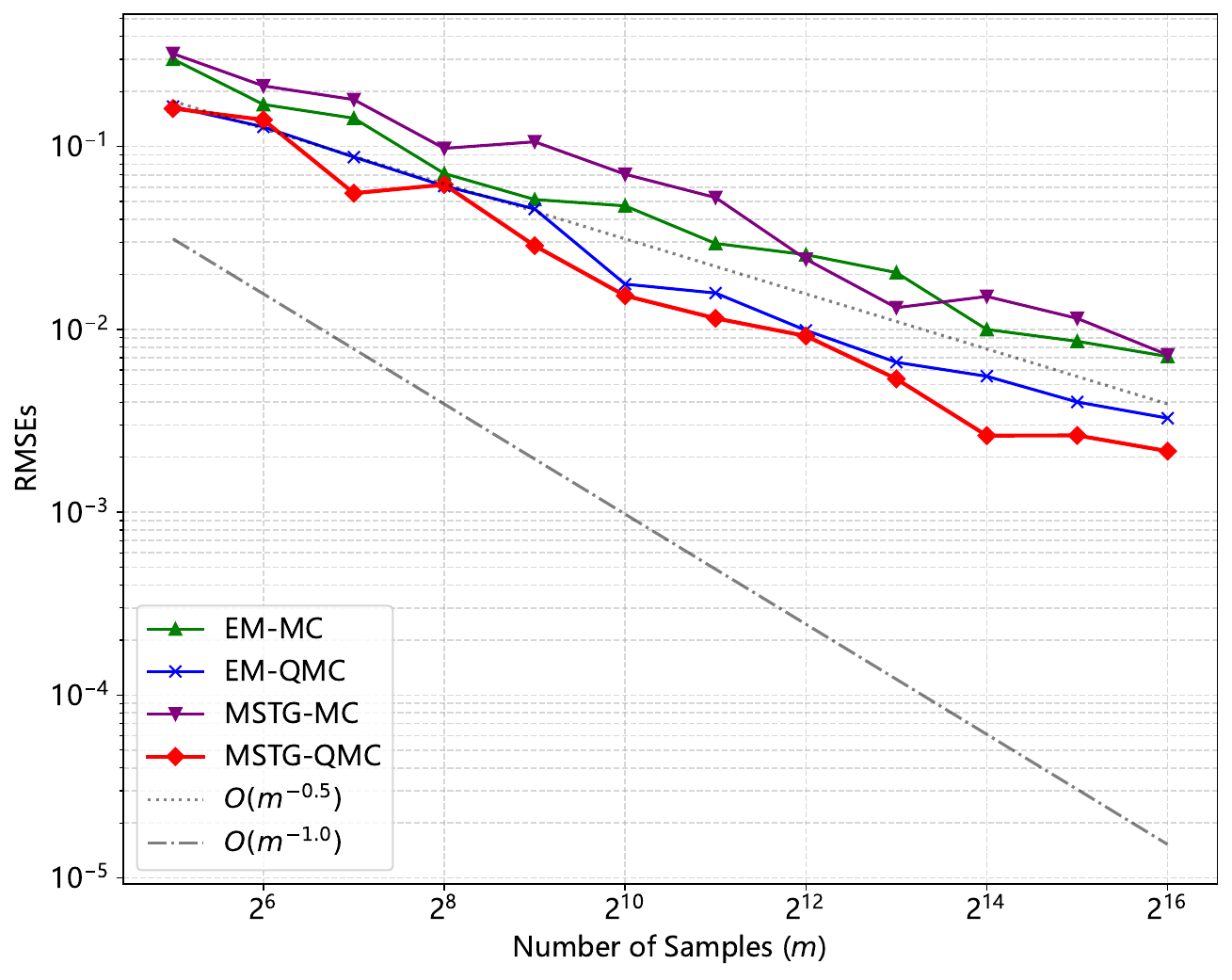}
    \caption{Case $d=4$: RMSE versus sample size for $N_{\mathrm{MSTG}}=5$ and $N_{\mathrm{EM}}=64$.}
    \label{fig:conv_high_dim_4d}
\end{figure}

\paragraph{The case $d=10$.}
We then use the same formulation \eqref{eq: numerical sde multidim}--\eqref{eq: numerical payoff multidim} with $d=10$.  The reference value, computed by a high-sample MLMC calculation, is $V_0^{(10)}=0.6788085524$.  For the MSTG estimator in this experiment, level sample sizes are allocated proportionally to the Poisson probabilities $p_n=e^{-1}/n!$.  In the RQMC version, a single Owen-scrambled Sobol' point set is generated in the total integration dimension of the largest retained MSTG level and then split across levels; each level uses the corresponding coordinate prefix.  This implementation preserves the Poisson level weighting and improves the empirical stability of the RQMC estimator.

The first comparison fixes the total integration dimension.  For EM, the total integration dimension is $dN_{\mathrm{EM}}$; for MSTG, the total integration dimension of the largest retained level is $(d+1)(N_{\mathrm{MSTG}}+1)$.  Thus choosing $N_{\mathrm{EM}}=11$ and $N_{\mathrm{MSTG}}=9$ gives the common total integration dimension $110$.  Table \ref{tab:dim10_same_dimension} shows that, under the same dimension budget, MSTG-QMC reduces the RMSE by more than a factor of four compared with EM-QMC, and requires less computational time.

\begin{table}[htbp]
\centering
\caption{Case $d=10$: same total integration dimension.}
\label{tab:dim10_same_dimension}
\begin{tabular}{lrrrrr}
\hline
Method & $m$ & $N$ & Int. dim. & RMSE & Time (s)\\
\hline
EM-MC & 65536 & 11 & 110 & $8.314{\times}10^{-3}$ & 4.66\\
EM-QMC & 65536 & 11 & 110 & $6.053{\times}10^{-3}$ & 8.91\\
MSTG-MC & 65536 & 9 & 110 & $5.849{\times}10^{-3}$ & 2.79\\
MSTG-QMC & 65536 & 9 & 110 & $\mathbf{1.426}{\times}\mathbf{10^{-3}}$ & \textbf{2.91}\\
\hline
\end{tabular}
\end{table}

The second comparison uses a refined EM discretization, $N_{\mathrm{EM}}=64$, while keeping the MSTG truncation level at $N_{\mathrm{MSTG}}=5$.  This reflects the different bias mechanisms: EM requires many deterministic time steps to reduce its algebraic discretization error, whereas the MSTG truncation error is already small at a low level.  With the same number of samples, Table \ref{tab:dim10_equal_samples} shows that MSTG-QMC attains a slightly smaller RMSE than EM-QMC while using less than one twentieth of the computational time.

\begin{table}[htbp]
\centering
\caption{Case $d=10$: equal sample size with comparable bias settings.}
\label{tab:dim10_equal_samples}
\begin{tabular}{lrrrrr}
\hline
Method & $m$ & $N$ & Int. dim. & RMSE & Time (s)\\
\hline
EM-MC & 65536 & 64 & 640 & $5.193{\times}10^{-3}$ & 26.56\\
EM-QMC & 65536 & 64 & 640 & $1.837{\times}10^{-3}$ & 50.67\\
MSTG-MC & 65536 & 5 & 66 & $7.950{\times}10^{-3}$ & 2.70\\
MSTG-QMC & 65536 & 5 & 66 & $\mathbf{1.685\times10^{-3}}$ & \textbf{2.57}\\
\hline
\end{tabular}
\end{table}

Finally, Table \ref{tab:dim10_same_time} compares the methods at comparable computational times.  The EM rows are taken from Table \ref{tab:dim10_equal_samples}.  Since MSTG is much cheaper per sample at the same bias level, the saved computational budget can be converted into a larger sample size.  In particular, MSTG-QMC with $2^{20}$ samples is still faster than EM-QMC with $2^{16}$ samples and reduces the RMSE by a factor of approximately $6.7$.

\begin{table}[htbp]
\centering
\caption{Case $d=10$: comparable computational times.}
\label{tab:dim10_same_time}
\begin{tabular}{lrrrrr}
\hline
Method & $m$ & $N$ & Int. dim. & RMSE & Time (s)\\
\hline
EM-MC & 131072 & 64 & 640 & $3.605{\times}10^{-3}$ & 49.56\\
MSTG-MC & 1310720 & 5 & 66 & $1.460{\times}10^{-3}$ & 49.83\\
EM-QMC & 65536 & 64 & 640 & $1.837{\times}10^{-3}$ & 50.67\\
MSTG-QMC & 1048576 & 5 & 66 & $\mathbf{2.758}{\times}\mathbf{10^{-4}}$ & \textbf{42.56}\\
\hline
\end{tabular}
\end{table}

Overall, the experiments support the theoretical message of the paper. RQMC improves the sampling error for both fixed-grid and stochastic-grid estimators, but the benefit is most pronounced when the estimator keeps the actual integration dimension small. The MSTG construction achieves this by replacing the slowly decaying EM bias with a super-exponentially small truncation error, so that high accuracy can be obtained without forcing RQMC into a very high-dimensional fixed-grid representation.

\section{Summary}
\label{sec:chap_sde_summary}

This paper studied RQMC methods for the numerical simulation of SDEs. For the Euler-Maruyama scheme, we proved that RQMC can improve the sampling error from the MC order $\mathcal{O}(m^{-1/2})$ to the nearly first-order rate $\mathcal{O}(m^{-1+\eps})$, where $\eps>0$ can be arbitrarily small, under a polynomial growth condition on the induced integrand.

The main message is that dimension reduction is essential for turning this high-order sampling behavior into practical efficiency. The proposed MSTG estimator reduces the actual integration dimension required by the SDE simulation: instead of forcing RQMC to work with high-dimensional fixed-grid Brownian inputs, it represents the problem through low-dimensional stochastic-grid levels and controls the remaining bias by a rapidly decaying truncation error. As a result, high accuracy can be achieved with far fewer input coordinates, so that QMC is applied to genuinely lower-dimensional integration problems. The numerical experiments in both one and multiple dimensions confirm that this reduction of the actual integration dimension is the key reason for the efficiency gain of the MSTG-QMC method.

\appendix
\section{Proof of Main Results}\label{sec: sde lemma proof}

\subsection{Proof of Theorem \ref{thm: partial deri growth}}\label{sec proof thm: partial deri growth}

\begin{proof}
By Assumption \ref{assum: qmc euler coef}, it holds that $ \widetilde X_k \in \mathcal{S}^N(\R^N),\ 0\le k\le N$. It remains to prove that for any $\bm u \subseteq 1{:}N$, the mixed partial derivative satisfies $ \partial^{\bm u}\widetilde X_k \in \mathcal{C}_p(\R^N)$. We proceed by induction on the order $|\boldsymbol{u}|$.

\textbf{Step 1: The base case $|\boldsymbol{u}|=0$ (growth of $\widetilde X_k$ itself).}

Recall the EM iteration:
\begin{equation}
    \widetilde X_k = \widetilde X_{k-1} + b(t_{k-1}, \widetilde X_{k-1})\Delta t + \sigma(t_{k-1}, \widetilde X_{k-1})\sqrt{\Delta t} z_k.
\end{equation}
Applying the triangle inequality and the linear growth conditions on the coefficients $b$ and $\sigma$, we have:
\begin{align}
    |\widetilde X_k| &\le |\widetilde X_{k-1}| + |b(t_{k-1}, \widetilde X_{k-1})|\Delta t + |\sigma(t_{k-1}, \widetilde X_{k-1})|\sqrt{\Delta t} |z_k| \notag \\
    &\le |\widetilde X_{k-1}| + C(1+|\widetilde X_{k-1}|)\Delta t + C(1+|\widetilde X_{k-1}|)\sqrt{\Delta t} |z_k|.
\end{align}
By adding 1 to both sides, we obtain the following recursive relation:
\begin{equation}
    (1+|\widetilde X_k|) \le (1+|\widetilde X_{k-1}|) \left( 1 + C\Delta t + C\sqrt{\Delta t}|z_k| \right).
\end{equation}
Iterating this inequality back to $\widetilde X_0$ yields:
\begin{equation}
\begin{aligned}
    (1+|\widetilde X_k|) &\le (1+|x_0|) \prod_{j=1}^k \left( 1 + C\Delta t + C\sqrt{\Delta t}|z_j| \right)\\
    &\le (|x_0|+ 1)\left( 1+ C\Delta t + C\sqrt{\Delta t}\frac{\sum_{j=1}^k |z_j|}{k}\right)^k\\
    &\le (|x_0| + 1) \left( 1+ C\Delta t + C\sqrt{\Delta t}\frac{|\bm z|}{\sqrt{k}}\right)^k.
    \end{aligned}
\end{equation}
Since the right-hand side is a polynomial in $|\bm z|$, it follows that $\widetilde X_k\in \mathcal{C}_p(\R^N)$.

\textbf{Step 2: The case $|\boldsymbol{u}|=1$.}

Suppose $\boldsymbol{u} = \{j\}$ for some $1 \le j \le k$. Let $J_{k, j} := \frac{\partial \widetilde X_k}{\partial \widetilde X_j}$ denote the state Jacobian. From the chain rule and the recursive structure of the EM scheme, we have:
\begin{equation}
    \frac{\partial \widetilde X_k}{\partial \widetilde X_{k-1}} = 1 + \partial_x b(t_{k-1}, \widetilde X_{k-1})\Delta t + \partial_x \sigma(t_{k-1}, \widetilde X_{k-1})\sqrt{\Delta t} z_k.
\end{equation}
For $k > j$, the derivative between state variables can be explicitly represented as:
\begin{equation} \label{eq:tangent_process}
    J_{k, j}:= \frac{\partial \widetilde X_k}{\partial \widetilde X_j}  = \prod_{i=j+1}^k \left( 1 + \partial_x b_{i-1} \Delta t + \partial_x \sigma_{i-1} \sqrt{\Delta t} z_i \right),
\end{equation}
where we use the shorthand $b_{i-1} = b(t_{i-1}, \widetilde X_{i-1})$ and similarly for $\sigma_{i-1}$, with the convention $J_{k, k} = 1$. We now examine the first-order partial derivative $\frac{\partial \widetilde X_k}{\partial z_j}$. By the chain rule:
\begin{equation}
    \frac{\partial \widetilde X_k}{\partial z_j} = \frac{\partial \widetilde X_k}{\partial \widetilde X_{j}} \frac{\partial \widetilde X_{j}}{\partial z_j} = J_{k,j}\sqrt{\Delta t}\sigma(t_{j-1},\widetilde X_{j-1}).
\end{equation}
The explicit expression for the first-order derivative is thus:
\begin{equation}\label{eq: J_kj}
    \frac{\partial \widetilde X_k}{\partial z_j} = \sigma(t_{j-1}, \widetilde X_{j-1})\sqrt{\Delta t} \prod_{m=j+1}^k \left( 1 + \partial_x b(t_{m-1},\widetilde X_{m-1})\Delta t + \partial_x\sigma(t_{m-1},\widetilde X_{m-1})\sqrt{\Delta t} z_m \right).
\end{equation}
Under Assumption \ref{assum: qmc euler coef}, we have $\partial_x b, \partial_x\sigma \in \mathcal{C}_p(\R)$ and $\sigma$ satisfies the linear growth condition. Combined with the result from Step 1 that $\widetilde X_m\in \mathcal{C}_p(\R^N)$, it follows that each term in \eqref{eq: J_kj} belongs to the function class $\mathcal{C}_p(\R^N)$. Thus, $ \frac{\partial \widetilde X_k}{\partial z_j} \in \mathcal{C}_p(\R^N)$.

\textbf{Step 3: Inductive Step for general order.}

Assume that for all $\boldsymbol{v} \subseteq 1{:}N$ such that $|\boldsymbol{v}| \le n$, and for all $1 \le k \le N$, we have $\partial^{\boldsymbol{v}} \widetilde X_k \in \mathcal{C}_p(\R^N)$. Now consider $\boldsymbol{u} = \boldsymbol{v} \cup \{i\}$ with $i > \max_{j \in \bm v} j$. By the commutativity of partial derivatives:
\begin{equation}
    \partial^{\boldsymbol{u}} \widetilde X_k = \partial^{\boldsymbol{v}} \partial_{z_i} \widetilde X_k = \partial^{\boldsymbol{v}} \left( \frac{\partial \widetilde X_k}{\partial \widetilde X_i} \frac{\partial \widetilde X_i}{\partial z_i} \right).
\end{equation}
Substituting the explicit expressions yields:
\begin{equation}\label{eq: partial x decomp}
        \begin{aligned}
            \partial^{\boldsymbol{u}} \widetilde X_k &= \partial^{\boldsymbol{v}} \left( J_{k,i} \cdot \sqrt{\Delta t} \sigma(t_{i-1}, \widetilde X_{i-1}) \right) \\
            &= \sqrt{\Delta t} \sum_{\boldsymbol{w} \subseteq \boldsymbol{v}} (\partial^{\boldsymbol{w}} J_{k,i}) (\partial^{\boldsymbol{v} \setminus \boldsymbol{w}} \sigma(t_{i-1}, \widetilde X_{i-1}))
        \end{aligned}
    \end{equation}
Since the function class $\mathcal{C}_p(\R^N)$ is closed under finite addition and multiplication, to prove that $\partial^{\boldsymbol{u}} \widetilde X_k \in \mathcal{C}_p(\R^N)$, it suffices to show that for any subset $\boldsymbol{w} \subseteq \boldsymbol{v}$, the two factors $\partial^{\boldsymbol{w}} J_{k,i}$ and $\partial^{\boldsymbol{v} \setminus \boldsymbol{w}} \sigma(t_{i-1}, \widetilde X_{i-1})$ in the summation both belong to $\mathcal{C}_p(\R^N)$.

    First, for the factor $\partial^{\boldsymbol{w}} J_{k,i}$, we have
    \begin{equation}
        \partial^{\boldsymbol{w}} J_{k,i} = \partial^{\boldsymbol{w}} \prod_{j=i+1}^k (1 + \partial_x b\Delta t + \partial_x \sigma\sqrt{\Delta t}z_j).
    \end{equation}
    Using the product rule for multi-indices, the right side of the above equation equals
    \begin{equation}
        \sum_{\boldsymbol{w}_{i+1} + \dots + \boldsymbol{w}_k = \boldsymbol{w}} \prod_{j=i+1}^k \partial^{\boldsymbol{w}_j} (1 + \partial_x b\Delta t + \partial_x \sigma\sqrt{\Delta t}z_j).
    \end{equation}
    Note that for each term $\partial^{\boldsymbol{w}_j} (1 + \partial_x b\Delta t + \partial_x \sigma\sqrt{\Delta t}z_j)$, according to Fa\`a di Bruno's formula, we obtain
    \begin{equation}
        \partial^{\boldsymbol{w}_j} (\partial_x b\Delta t + \partial_x \sigma\sqrt{\Delta t}z_j) = \sum_{\lambda \in \pi(\boldsymbol{w}_j)} \left(\Delta t \partial_x^{1+|\lambda|} b+ \sqrt{\Delta t} z_j \partial_x^{1+|\lambda|}\sigma\right)\prod_{\eta \in \lambda} \partial^\eta \widetilde X_{j-1},
    \end{equation}
    where $ \pi(\boldsymbol w_j)$ is the set of all partitions of $ \boldsymbol{w}_j$, and $\lambda$ is one such partition, satisfying $ \sum_{\eta\in \lambda} \eta = \boldsymbol{w}_j$. Since $\eta \subseteq \boldsymbol{w}_j \subseteq \boldsymbol{w} \subseteq \boldsymbol{v}$, by the induction hypothesis, we have $\partial^\eta \widetilde X_j \in \mathcal{C}_p(\R^N)$. Moreover, since all derivatives of $b$ and $\sigma$ belong to $ \mathcal{C}_p(\R^N)$, and $\widetilde X_j \in \mathcal{C}_p(\R^N)$, we have $\partial^{\boldsymbol{w}} J_{k,i} \in \mathcal{C}_p(\R^N)$.

    Second, for the factor $\partial^{\boldsymbol{v} \setminus \boldsymbol{w}} \sigma(t_{i-1}, \widetilde X_{i-1})$, using Fa\`a di Bruno's formula again, we have
    \begin{equation}
        \partial^{\boldsymbol{v} \setminus \boldsymbol{w}} \sigma(t_{i-1}, \widetilde X_{i-1}) = \sum_{\lambda \in \pi(\boldsymbol{v} \setminus \boldsymbol{w})} \partial_x^{|\lambda|}\sigma \prod_{\eta \in \lambda} \partial^\eta \widetilde X_{i-1}.
    \end{equation}
    Similarly, by the induction hypothesis, we have $\partial^\eta \widetilde X_{i-1}\in \mathcal{C}_p(\R^N)$, and $\partial_x^{|\lambda|}\sigma \in \mathcal{C}_p(\R^N)$. Therefore, this factor also belongs to the function class $ \mathcal{C}_p(\R^N)$.

    In summary, for any $\boldsymbol{w} \subseteq \boldsymbol{v}$, the factors in the summation on the right-hand side of equation \eqref{eq: partial x decomp} all belong to $ \mathcal{C}_p(\R^N)$, hence $\partial^{\boldsymbol{u}} \widetilde X_k \in \mathcal{C}_p(\R^N)$.
    By mathematical induction, the theorem is proved.
\end{proof}

\subsection{Proof of Lemma \ref{lemma: mstg truncation error}}\label{sec proof lemma: mstg truncation error}

\begin{proof}
    First, for an arbitrary $n$, we estimate the upper bound of $|\mathbb{E}[\widehat{\psi} \mid N_T = n]|$. From the boundedness of the drift coefficient $b$, it follows that:
    \begin{equation}
    \begin{aligned}
        \E\left[ |\widehat X_{T_i} - \widehat X_{T_{i-1}}| \mid T_1, \dots, T_i \right] &\le \E\left[ | b(T_{i-1}, \widehat X_{T_{i-1}}) \Delta T_i + \sigma_0 \Delta W_{T_i} | \mid T_1, \dots, T_i \right] \\
        &\le C |\Delta T_i|^{1/2}.
    \end{aligned}
    \end{equation}
    Recall that the drift $b$ is uniformly $1/2$-H\"older continuous in $t$ and Lipschitz continuous in $x$, and the terminal function $g$ satisfies the global Lipschitz condition. For $n \ge 1$, we have:
    \begin{equation}\label{eq: upper bound of psi_n}
        \begin{aligned}
            |\E[\widehat{\psi}\mid N_T = n]| &=\E\left[ e^{\beta T}\left[ g(\widehat{X}_T) - g(\widehat{X}_{T_{n}})\right] \prod_{i=1}^{n} \mathcal{W}_i\Bigg|N_T = n\right]\\
            &\le e^{\beta T}\beta^{-n} \E\left[ |\Delta T_{n+1}|^{1/2} \prod_{i=1}^n \frac{|\Delta T_i|^{1/2}|\Delta T_{i+1}|^{1/2}}{|\Delta T_{i+1}|}\Bigg| N_T = n\right]\\
            &\le Ce^{\beta T}\beta^{-n}\E\left[ |\Delta T_1|^{1/2}\big| N_T = n\right] \\
            &\le C T^{1/2}e^{\beta T} \beta^{-n}.
        \end{aligned}
    \end{equation}
    Substituting this bound back into the summation for the truncation error, we obtain:
    \begin{equation}
    \begin{aligned}
        \sum_{n = N+1}^{\infty}p_n \left| \E[\widehat \psi\mid N_T = n]\right| &\le \sum_{n=N+1}^{\infty} \frac{(\beta T)^n}{n!} e^{-\beta T} CT^{1/2}e^{\beta T} \beta^{-n}\le \sum_{n = N+1}^\infty \frac{T^n}{n!} CT^{1/2}.
        \end{aligned}
    \end{equation}
    It remains to estimate the upper bound of the summation on the right-hand side. Let $Y$ be a Poisson random variable with intensity (parameter) $T$. We observe that:
    \begin{equation}
        \begin{aligned}
            \sum_{n = N+1}^\infty \frac{T^n}{n!} &= e^{T}\sum_{n = N+1}^\infty \frac{T^n}{n!} e^{-T} = e^{T} \p\left( Y>N\right)\\
            &\le e^{T}\frac{\E[e^{tY}]}{e^{tN}}  = e^{T}\frac{e^{T(e^t - 1)}}{e^{tN}}\\
            &\le e^{T}e^{T(N/T -1) - N\log(N/T)}\\
            &\le (N/(eT))^{-N},
        \end{aligned}
    \end{equation}
    where we utilized the Chernoff bound (a specific case of Chebyshev's inequality) in the first inequality, and set $t = \log(N/T)$ in the second inequality. This tail estimate directly yields the desired result.
\end{proof}

\subsection{Proof of Lemma \ref{lemma: deri delta t growth}}\label{sec proof lemma: deri delta t growth}

\begin{proof}
We first consider \textbf{Case 1}: $i \notin \boldsymbol{u}^{\tau}$. By the chain rule, there exists a constant $C$ such that:
\begin{equation}\label{eq: partial delta t}
    \begin{aligned}
        \left|\partial^{\boldsymbol{u}^{\tau}}_{\boldsymbol{z}^{\tau}} \Delta T_i\right| &= T\left|\partial^{\boldsymbol{u}^{\tau}}_{\boldsymbol{z}^{\tau}} \frac{F_{\tau}^{-1}\circ \Phi(z_i^{\tau})}{\sum_{j=1}^{n+1} F_{\tau}^{-1}\circ \Phi(z_j^{\tau})}\right|\\
        & = C\left( \sum_{j=1}^{n+1} F_{\tau}^{-1}\circ \Phi(z_j^{\tau})\right)^{-(|\boldsymbol{u}^{\tau}| +1)}F_{\tau}^{-1}\circ \Phi(z_i^{\tau})\prod_{j\in \boldsymbol{u}^{\tau}} \frac{\varphi(z_j^{\tau})}{1-\Phi(z_j^{\tau})}\\
        &\le C\Delta T_i \prod_{j\in\boldsymbol{u}^{\tau}}\left( \sum_{j=1}^{n+1} F_{\tau}^{-1}\circ \Phi(z_j^{\tau})\right)^{-1} \frac{\varphi(z_j^{\tau})}{1-\Phi(z_j^{\tau})}.
    \end{aligned}
\end{equation}
It suffices to bound each factor in the product on the right-hand side. Recalling that $F_{\tau}^{-1}(x) = -\log(1-x)$ and that for any $x\in (0,1)$, $x \le - \log(1-x) \le \frac{x}{1-x}$, we have:
\begin{equation}\label{eq: partial delta t prod mid term}
    \left( \sum_{j=1}^{n+1} F_{\tau}^{-1}\circ \Phi(z_j^{\tau})\right)^{-1} \frac{\varphi(z_j^{\tau})}{1-\Phi(z_j^{\tau})} \le \frac{1}{F_{\tau}^{-1}\circ\Phi(z_j^{\tau})} \frac{\varphi(z_j^{\tau})}{1-\Phi(z_j^{\tau})} \le \frac{1}{\Phi(z_j^{\tau})}\frac{\varphi(z_j^{\tau})}{1-\Phi(z_j^{\tau})} .
\end{equation}
We analyze the upper bound in \eqref{eq: partial delta t prod mid term} by considering the limits $z_j^{\tau} \to \pm \infty$.

When $z_j^{\tau} \to -\infty$, assume without loss of generality that $z_j^{\tau} < -2$. In this regime, $\Phi(z_j^{\tau}) \to 0$, so $\frac{1}{1-\Phi(z_j^{\tau})} \to 1$. We then only need to bound the ratio $\frac{\varphi(z_j^{\tau})}{\Phi(z_j^{\tau})}$. By the Mills ratio, we have:
\begin{equation}
    \frac{\varphi(z_j^{\tau})}{\Phi(z_j^{\tau})} \le |z_j^{\tau}| + \frac{1}{|z_j^{\tau}|} \le |z_j^{\tau}| + \frac{1}{2} \le 2|z_j^{\tau}|.  
\end{equation}
When $z_j^{\tau} \to +\infty$, assume without loss of generality that $z_j^{\tau} > 2$. Here $\Phi(z_j^{\tau}) \to 1$, so $\frac{1}{\Phi(z_j^{\tau})} \to 1$. We then only need to bound $\frac{\varphi(z_j^{\tau})}{1-\Phi(z_j^{\tau})}$. Again, by Mills ratio, we have:
\begin{equation}
    \frac{\varphi(z_j^{\tau})}{1-\Phi(z_j^{\tau})} \le z_j^{\tau} + \frac{1}{z_j^{\tau}} \le z_j^{\tau} + \frac{1}{2} \le 2z_j^{\tau}.  
\end{equation}
Combining these scenarios with \eqref{eq: partial delta t} and \eqref{eq: partial delta t prod mid term}, we conclude that there exists a constant $C$ such that:
\begin{equation}
    \left| \partial^{\boldsymbol u^{\tau}}_{\boldsymbol z^{\tau}}\Delta T_i\right| \le C\Delta T_i \prod_{j \in \boldsymbol{u}^\tau}|z_j^{\tau}|.
\end{equation}

For \textbf{Case 2}: $i \in \boldsymbol{u}^{\tau}$, we first establish the following upper bound:
\begin{equation}
    \begin{aligned}
        \left|\partial^{\boldsymbol{u}^{\tau}}_{\boldsymbol{z}^{\tau}} \Delta T_i\right| &= T\left|\partial^{\boldsymbol{u}^{\tau}\setminus\{i\}}_{\boldsymbol{z}^{\tau}}\partial_{z_i^{\tau}} \frac{F_{\tau}^{-1}\circ \Phi(z_i^{\tau})}{\sum_{j=1}^{n+1} F_{\tau}^{-1}\circ \Phi(z_j^{\tau})}\right|\\
        & = T\left| \partial^{\boldsymbol{u}^{\tau}\setminus\{i\}}_{\boldsymbol{z}^{\tau}} \frac{\sum_{j\ne i} F_{\tau}^{-1}\circ \Phi(z_j^{\tau})}{\left(\sum_{j=1}^{n+1} F_{\tau}^{-1}\circ \Phi(z_j^{\tau})\right)^2}\frac{\varphi(z_i^{\tau})}{1 - \Phi(z_i^{\tau})}\right|\\
        & = T\left|\frac{\varphi(z_i^{\tau})}{1 - \Phi(z_i^{\tau})}\partial^{\boldsymbol{u}^{\tau}\setminus\{i\}}_{\boldsymbol{z}^{\tau}} \left( \frac{1}{\sum_{j=1}^{n+1}F_{\tau}^{-1}\circ \Phi(z_j^{\tau})} - \frac{F_{\tau}^{-1}\circ \Phi(z_i^{\tau})}{\left( \sum_{j=1}^{n+1}F_{\tau}^{-1}\circ \Phi(z_j^{\tau})\right)^2}\right)\right|\\
        &\le T\frac{\varphi(z_i^{\tau})}{1 - \Phi(z_i^{\tau})} \frac{1}{F_{\tau}^{-1}\circ \Phi(z_i^{\tau})} \left|\partial^{\boldsymbol{u}^{\tau}\setminus\{i\}}_{\boldsymbol{z}^{\tau}} \Delta T_i \right| \\
        &\qquad+ T\frac{\varphi(z_i^{\tau})}{1 - \Phi(z_i^{\tau})} \left| \partial^{\boldsymbol{u}^{\tau}\setminus\{i\}}_{\boldsymbol{z}^{\tau}} \frac{F_{\tau}^{-1}\circ \Phi(z_i^{\tau})}{\left( \sum_{j=1}^{n+1}F_{\tau}^{-1}\circ \Phi(z_j^{\tau})\right)^2}\right|.
    \end{aligned}
\end{equation}
Following the analysis used in Case 1, the first term on the right-hand side is bounded by $C\Delta T_i \prod_{j\in\boldsymbol{u}^{\tau}}|z_j^{\tau}|$. For the second term, applying the chain rule yields:
\begin{equation}\notag
\begin{aligned}
    &T\frac{\varphi(z_i^{\tau})}{1 - \Phi(z_i^{\tau})} \left| \partial^{\boldsymbol{u}^{\tau}\setminus\{i\}}_{\boldsymbol{z}^{\tau}} \frac{F_{\tau}^{-1}\circ \Phi(z_i^{\tau})}{\left( \sum_{j=1}^{n+1}F_{\tau}^{-1}\circ \Phi(z_j^{\tau})\right)^2}\right| \\
    &\le C \frac{\varphi(z_i^{\tau})}{1 - \Phi(z_i^{\tau})} \left( \sum_{j=1}^{n+1}F_{\tau}^{-1}\circ \Phi(z_j^{\tau})\right)^{-1}\Delta T_i \prod_{j\in \boldsymbol{u}^{\tau}\setminus \{i\}} \left( \sum_{j=1}^{n+1}F_{\tau}^{-1}\circ \Phi(z_j^{\tau})\right)^{-1}\frac{\varphi(z_j^{\tau})}{1-\Phi(z_j^{\tau})}\\
    &\le C \Delta T_i \prod_{j\in \bm u^{\tau}} |z_j^{\tau}|,
    \end{aligned}
\end{equation}
where the final inequality is obtained via the same methodology as Case 1. This completes the proof of the lemma.
\end{proof}

\subsection{Proof of Lemma \ref{lemma: deri x growth}}\label{sec proof lemma: deri x growth}

\begin{proof}
We establish the lemma in several steps.

\noindent\textbf{Step 1: $\widehat{X}_{T_k} \in \mathcal{C}_p(\R^{2n+2})$.} \\
Under Assumption \ref{assum: smooth drift}, the numerical discretization scheme can be expressed as:
\begin{equation} \label{eq:euler_scheme}
    \widehat{X}_{T_k} = \widehat{X}_{T_{k-1}} + b( \widehat{X}_{T_{k-1}}) \Delta T_k + \sigma_0 \sqrt{\Delta T_k} z_k^W \text{}.
\end{equation}
The drift coefficient $b$ satisfies the linear growth condition, i.e., there exists a constant $C$ such that $|b(x)| \le C(1+|x|)$. Combined with \eqref{eq:euler_scheme}, we obtain the following bound for the state variable:
\begin{equation} \label{eq:state_bound_step1}
    |\widehat{X}_{T_k}| \le |\widehat{X}_{T_{k-1}}| + C(1+|\widehat{X}_{T_{k-1}}|)\Delta T_k + |\sigma_0| \sqrt{\Delta T_k} |z_k^W| \text{}.
\end{equation}
By adding $1$ to both sides of \eqref{eq:state_bound_step1} and factoring out $(1+|\widehat{X}_{T_{k-1}}|)$, we have:
\begin{equation} \label{eq:state_bound_step2}
\begin{aligned}
    (1+|\widehat{X}_{T_k}|) &\le (1+|\widehat{X}_{T_{k-1}}|) + C(1+|\widehat{X}_{T_{k-1}}|)\Delta T_k + |\sigma_0| \sqrt{\Delta T_k} |z_k^W| \\
    &\le (1+|\widehat{X}_{T_{k-1}}|) \left( 1 + C\Delta T_k + |\sigma_0| \sqrt{\Delta T_k} |z_k^W| \right) \text{}.
\end{aligned}
\end{equation}
Iteratively expanding the inequality \eqref{eq:state_bound_step2} back to the initial value $\widehat{X}_0$, and applying the generalized H\"older inequality, we obtain:
\begin{equation} \label{eq:state_bound_final}
    \begin{aligned}
    (1+|\widehat{X}_{T_k}|) &\le (1+|\widehat{X}_0|) \prod_{j=1}^{k} \left( 1 + C\Delta T_j + |\sigma_0| \sqrt{\Delta T_j} |z_j^W| \right)\\
    &\le (1+|\widehat{X}_0|)  \left( 1 + C T/k + |\sigma_0|\sum_{j=1}^k \sqrt{\Delta T_j} |z_j^W|/k \right)^k\\
    &\le (1+|\widehat{X}_0|)  \left( 1 + C T/k + |\sigma_0| \frac{\sqrt{T}}{k} |\bm z^W| \right)^k.
    \end{aligned}
\end{equation}
The expression in \eqref{eq:state_bound_final} ensures that $\widehat{X}_{T_k} \in \mathcal{C}_p(\R^{2n+2})$.

\noindent\textbf{Step 2: $\partial_{\bm z^\tau}^{\bm u^{\tau}}\widehat{X}_{T_k} \in \mathcal{C}_p(\R^{2n+2})$.} \\
We proceed by induction on the index $k$ to prove that for any $\bm u^{\tau} \subseteq 1{:}(n+1)$ and $0 \le k \le n+1$, the partial derivative $\partial_{\bm z^\tau}^{\bm u^{\tau}}\widehat{X}_{T_k} \in \mathcal{C}_p(\R^{2n+2})$. For $k = 0$, $\partial_{\bm z^\tau}^{\bm u^{\tau}}\widehat{X}_{T_0} = 0$, thus the base case holds.

Suppose the hypothesis holds for $k-1$; i.e., $\partial_{\bm z^\tau}^{\bm u^{\tau}}\widehat{X}_{T_{k-1}} \in \mathcal{C}_p(\R^{2n+2})$ for any $\bm u^{\tau} \subseteq 1{:}(n+1)$. Considering $\partial_{\bm z^\tau}^{\bm u^{\tau}}\widehat{X}_{T_k}$, the chain rule yields:
\begin{equation} \label{eq:first_derivative}
    \partial_{\bm z^\tau}^{\bm u^{\tau}}\widehat{X}_{T_k} =  \partial_{\bm z^\tau}^{\bm u^{\tau}}\widehat{X}_{T_{k-1}} + \sum_{\bm u_1^{\tau} + \bm u_2^{\tau} = \bm u^{\tau}} \partial_{\bm z^\tau}^{\bm u^{\tau}_1}b(\widehat{X}_{T_{k-1}})\partial_{\bm z^\tau}^{\bm u^{\tau}_2}\Delta T_k + \sigma_0\partial_{\bm z^\tau}^{\bm u^{\tau}}\sqrt{\Delta T_k}z_k^W.
\end{equation}
We analyze each term on the right-hand side of \eqref{eq:first_derivative} individually. For the first term, the inductive hypothesis implies $\partial_{\bm z^\tau}^{\bm u^{\tau}}\widehat{X}_{T_{k-1}} \in \mathcal{C}_p(\R^{2n+2})$. For $\partial_{\bm z^\tau}^{\bm u^{\tau}_1} b(\widehat{X}_{T_{k-1}})$ in the second term, we apply the Fa\`{a} di Bruno formula:
\begin{equation}\label{eq: drift derivatives fa}
    \partial_{\bm z^\tau}^{\bm u^{\tau}_1}b(\widehat{X}_{T_{k-1}}) = \sum_{w\in \pi(\bm u_1^{\tau})} b^{(|w|)}(\widehat{X}_{T_{k-1}}) \prod_{\lambda \in w} \partial^{\lambda}_{\bm z^{\tau}} \widehat{X}_{T_{k-1}}.
\end{equation}
Since all derivatives of $b$ are bounded by assumption, and the inductive hypothesis ensures $\partial^{\lambda}_{\bm z^{\tau}} \widehat{X}_{T_{k-1}} \in \mathcal{C}_p(\R^{2n+2})$, it follows that $\partial_{\bm z^\tau}^{\bm u^{\tau}_1} b(\widehat{X}_{T_{k-1}}) \in \mathcal{C}_p(\R^{2n+2})$. Together with the bound for $\partial_{\bm z^\tau}^{\bm u^{\tau}_2}\Delta T_k$ from Lemma \ref{lemma: deri delta t growth}, the entire second term belongs to $\mathcal{C}_p(\R^{2n+2})$.

For the derivative of the square root term in the third term, applying the Fa\`{a} di Bruno formula and Lemma \ref{lemma: deri delta t growth} again gives:
\begin{equation} \label{eq:sqrt_derivative}
    \begin{aligned}
    \left|\partial_{\bm z^\tau}^{\bm u^{\tau}}\sqrt{\Delta T_k}\right| &= \left|\sum_{w\in \pi(\bm u^{\tau})} (-1)^{|w|-1}\frac{(2|w|-3)!!}{2^{|w|}} (\Delta T_k)^{\frac{1}{2}- |w|} \prod_{\lambda \in w} \partial^{\lambda}_{\bm z^{\tau}} \Delta T_{k}\right| \\
    &\le C \sum_{w\in \pi(\bm u^{\tau})} (\Delta T_k)^{\frac{1}{2}- |w|} (\Delta T_k)^{|w|} \prod_{\lambda \in w} \prod_{i \in \lambda} |z_i^{\tau}|\\
    & = C \sum_{w\in \pi(\bm u^{\tau})} (\Delta T_k)^{\frac{1}{2}}  \prod_{\lambda \in w}\prod_{i \in \lambda} |z_i^{\tau}|,
    \end{aligned}
\end{equation}
where we denote $ (-1)!! = 1$. Combining \eqref{eq:first_derivative}, \eqref{eq: drift derivatives fa}, and \eqref{eq:sqrt_derivative}, we conclude $\partial_{\bm z^\tau}^{\bm u^{\tau}}\widehat{X}_{T_k} \in \mathcal{C}_p(\R^{2n+2})$. By mathematical induction, the mixed partial derivative satisfies $\partial_{\bm z^\tau}^{\bm u^{\tau}}\widehat{X}_{T_k} \in \mathcal{C}_p(\R^{2n+2})$ for all $\bm u^{\tau}$ and $k$.

\noindent\textbf{Step 3: $\partial^{\bm u^{\tau}}_{\bm z^{\tau}}\partial^{\bm u^W}_{\bm z^W} \widehat{X}_{T_k} \in \mathcal{C}_p(\R^{2n+2})$.} \\
We now proceed by induction on the order of derivatives with respect to $\bm{z}^W$, denoted by $m = |\bm{u}^W|$, to prove the conclusion for all $\bm u^{\tau}, \bm u^W \subseteq 1{:}(n+1)$. From the previous steps, the base case $m=0$ is established.

Assuming the conclusion holds for all $|\bm u^W| \le m$, we consider the case for $m+1$. For any $1 \le j \le n+1$, we compute the derivative $\partial^{\bm u^{\tau}}_{\bm z^{\tau}} \partial^{\bm u^W \cup \{j\}}_{\bm z^W} \widehat{X}_{T_k}$:
\begin{equation}\label{eq: deri x}
    \begin{aligned}
        \partial^{\bm u^{\tau}}_{\bm z^{\tau}}\partial^{\bm u^W\cup\{j\}}_{\bm z^W} \widehat{X}_{T_k} &= \partial^{\bm u^{\tau}}_{\bm z^{\tau}}\partial^{\bm u^W}_{\bm z^W} \frac{\partial\widehat{X}_{T_k}}{\partial z_j^W} = \partial^{\bm u^{\tau}}_{\bm z^{\tau}}\partial^{\bm u^W}_{\bm z^W} \frac{\partial\widehat{X}_{T_k}}{\partial\widehat{X}_{T_j}}\frac{\partial\widehat{X}_{T_j}}{\partial z_j^W}\\
        & = \partial^{\bm u^{\tau}}_{\bm z^{\tau}}\partial^{\bm u^W}_{\bm z^W}\left( J_{k,j}\sigma_0 \sqrt{\Delta T_j}\right)\\
        & = \sigma_0\sum_{\bm u_1^{\tau} + \bm u_2^{\tau} = \bm u^{\tau}} \partial^{\bm u_1^{\tau}}_{\bm z^{\tau}}\partial^{\bm u^W}_{\bm z^W} J_{k,j} \partial^{\bm u_2^{\tau}}_{\bm z^{\tau}}\sqrt{\Delta T_j},
    \end{aligned}
\end{equation}
where the state Jacobian $J_{k,j}$ is defined as:
\begin{equation}
    J_{k,j}: =\frac{\partial \widehat X_{T_k}}{\partial\widehat{X}_{T_j}} = \prod_{i = j+1}^k (1 + \frac{\dd }{\dd x} b(\widehat{X}_{T_{i-1}})\Delta T_i).
\end{equation}
By \eqref{eq:sqrt_derivative}, we have $\partial^{\bm u_2^{\tau}}_{\bm z^{\tau}} \sqrt{\Delta T_j} \in \mathcal{C}_p(\R^{n+1})$. Thus, it suffices to prove $\partial^{\bm u_1^{\tau}}_{\bm z^{\tau}} \partial^{\bm u^W}_{\bm z^W} J_{k,j} \in \mathcal{C}_p(\R^{2n+2})$. Using the chain rule:
\begin{equation}\label{eq: deri J}
\begin{aligned}
    \partial^{\bm u_1^{\tau}}_{\bm z^{\tau}}\partial^{\bm u^W}_{\bm z^W} J_{k,j} &= \partial^{\bm u_1^{\tau}}_{\bm z^{\tau}}\partial^{\bm u^W}_{\bm z^W}\prod_{i = j+1}^k\left(1+ \frac{\dd }{\dd x}b(\widehat X_{T_{i-1}})\Delta T_i \right)\\
    & = \sum_{\bm v^{\tau}_{j+1}+\dots+ \bm v_{k}^{\tau} = \bm u^{\tau}_1} \sum_{\bm v^{W}_{j+1}+\dots+ \bm v_{k}^{W} = \bm u^{W}} \prod_{i = j+1}^k \partial_{\bm z^{\tau}}^{\bm v^{\tau}_i} \partial_{\bm z^W}^{\bm v_i^W}\left(1+ \frac{\dd }{\dd x}b(\widehat X_{T_{i-1}})\Delta T_i \right).
\end{aligned}
\end{equation}
Under the inductive hypothesis, and utilizing the boundedness of the derivatives of $b$ and Lemma \ref{lemma: deri delta t growth}, it follows that each factor in \eqref{eq: deri J} satisfies $\partial_{\bm z^{\tau}}^{\bm v^{\tau}_i} \partial_{\bm z^W}^{\bm v_i^W} \frac{\dd}{\dd x} b(\widehat X_{T_{i-1}}) \Delta T_i \in \mathcal{C}_p(\R^{2n+2})$. Consequently, combining \eqref{eq: deri x} and \eqref{eq: deri J}, we have $\partial^{\bm u^{\tau}}_{\bm z^{\tau}} \partial^{\bm u^W \cup \{j\}}_{\bm z^W} \widehat{X}_{T_k} \in \mathcal{C}_p(\R^{2n+2})$.

By the principle of induction, the mixed partial derivatives $\partial^{\bm u^{\tau}}_{\bm z^{\tau}} \partial^{\bm u^W}_{\bm z^W} \widehat{X}_{T_k}$ belong to $\mathcal{C}_p(\R^{2n+2})$ for all indices $\bm u^{\tau}, \bm u^W$ and $k$. This completes the proof.
\end{proof}

\subsection{Proof of Lemma \ref{lemma: deri diff growth}}\label{sec proof lemma: deri diff growth}

\begin{proof}
    We begin by applying the Leibniz rule to the expression in \eqref{eq: diff target}. Note that $(\Delta T_i)^{-1/2}$ depends only on $\bm z^\tau$, whereas the operator $\partial^{\bm u^W}_{\bm z^W}$ acts only on the function difference:
    \begin{equation} \label{eq: diff target leibniz}
    \begin{aligned}
        \partial^{\bm u^{\tau}}_{\bm z^{\tau}}\partial^{\bm u^W}_{\bm z^W}\left[ \frac{f(\widehat{X}_{T_i}) - f(\widehat{X}_{T_{i-1}})}{\sqrt{\Delta T_i}} \right] &= \sum_{\bm v \subseteq \bm u^{\tau}} 
        \Bigl[ \partial^{\bm u^{\tau} \setminus \bm v}_{\bm z^{\tau}} (\Delta T_i)^{-\frac{1}{2}} \Bigr] \\
        &\qquad \times \Bigl[ \partial^{\bm v}_{\bm z^{\tau}} \partial^{\bm u^W}_{\bm z^W} \bigl( f(\widehat{X}_{T_i}) - f(\widehat{X}_{T_{i-1}}) \bigr) \Bigr].
    \end{aligned}
    \end{equation}

    For the first factor on the right-hand side of \eqref{eq: diff target leibniz}, an expansion via the Fa\`{a} di Bruno formula yields:
    \begin{equation} \label{eq:time_part}
    \begin{aligned}
        \left|\partial^{\bm u^{\tau} \setminus \bm v}_{\bm z^{\tau}} (\Delta T_i)^{-\frac{1}{2}}\right| 
        &= \left|\sum_{w \in \pi(\bm u^{\tau} \setminus \bm v)} 
        C_{|w|} \, (\Delta T_i)^{-\frac{1}{2}-|w|} 
        \prod_{\lambda \in w} \partial^{\lambda}_{\bm z^{\tau}} (\Delta T_i)\right|\\
        & \le C \sum_{w \in \pi(\bm u^{\tau} \setminus \bm v)} (\Delta T_i)^{-\frac{1}{2} }  \prod_{\lambda \in w} \prod_{i \in \lambda} |z_i^{\tau}|
        \end{aligned}
    \end{equation}
    where $C_{|w|} = (-1)^{|w|} \frac{(2|w|-1)!!}{2^{|w|}}$. The final inequality follows from the growth bounds for $\Delta T_i$ established in Lemma \ref{lemma: deri delta t growth}.

    By the mixed version of the Fa\`{a} di Bruno formula, the term $\partial^{\bm v}_{\bm z^{\tau}} \partial^{\bm u^W}_{\bm z^W} f(\widehat{X}_{T_j})$ for $j=i$ or $i-1$ can be expressed as:
    \begin{equation} \label{eq: faa di mix f}
        \partial^{\bm v}_{\bm z^{\tau}} \partial^{\bm u^W}_{\bm z^W} f(\widehat{X}_{T_j})
        = \sum_{w \in \pi(\bm v \sqcup \bm u^W)} 
        f^{(|w|)}(\widehat{X}_{T_j}) 
        \prod_{D \in w} 
        \Bigl( \partial^{D \cap \bm v}_{\bm z^{\tau}} 
               \partial^{D \cap \bm u^W}_{\bm z^W} \widehat{X}_{T_j} \Bigr).
    \end{equation}
    where $\sqcup$ denotes the disjoint union, treating $\bm v$ and $\bm u^W$ as distinct sets of indices. To establish the lemma, it remains to prove that the following operator belongs to the function class $\mathcal{C}_p(\R^{2n+2})$:
    \begin{equation} \label{eq:diff_operator_I}
        \mathcal{I} := \left( \Delta T_i\right)^{-1/2}\partial^{\bm v}_{\bm z^{\tau}}\partial^{\bm u^W}_{\bm z^W}\left( f(\widehat{X}_{T_{i-1}} + \Delta X_i) -  f(\widehat{X}_{T_{i-1}})\right).
    \end{equation}
    Substituting \eqref{eq: faa di mix f} into \eqref{eq:diff_operator_I}, we obtain:
    \begin{equation} \label{eq:diff_I_expansion}
    \begin{aligned}
        \mathcal{I} = \sum_{w \in \pi(\bm v \sqcup \bm u^W)} 
             \left( \Delta T_i\right)^{-1/2}\bigg(& f^{(|w|)}(\widehat{X}_{T_{i-1}} + \Delta X_i) \prod_{D \in w} 
             \partial^{D \cap \bm v}_{\bm z^{\tau}} 
                   \partial^{D \cap \bm u^W}_{\bm z^W}   (\widehat{X}_{T_{i-1}} + \Delta X_i) \\
        &- f^{(|w|)}(\widehat{X}_{T_{i-1}}) \prod_{D \in w} 
            \Bigl( \partial^{D \cap \bm v}_{\bm z^{\tau}} 
                   \partial^{D \cap \bm u^W}_{\bm z^W} \widehat{X}_{T_{i-1}} \Bigr) \bigg),\\
        \end{aligned}
    \end{equation}
    where $ f^{(|w|)}$ is the $ |w|$-th derivative of $ f$. We show that each difference term in the summation belongs to $\mathcal{C}_p(\R^{2n+2})$. By adding and subtracting the cross-term $f^{(|w|)}(\widehat{X}_{T_{i-1}}) \prod_{D \in w} \partial^{D \cap \bm v}_{\bm z^{\tau}} \partial^{D \cap \bm u^W}_{\bm z^W} \widehat{X}_{T_i}$ within the summation, we have:
\begin{equation} \label{eq:diff_I_decoupled}
\begin{aligned}
     &\left( \Delta T_i\right)^{-1/2}\left( f^{(|w|)}(\widehat{X}_{T_{i}}) - f^{(|w|)}(\widehat{X}_{T_{i-1}}) \right) \prod_{D \in w} 
         \partial^{D \cap \bm v}_{\bm z^{\tau}} 
               \partial^{D \cap \bm u^W}_{\bm z^W}   \widehat{X}_{T_{i}} \\
    &+ f^{(|w|)}(\widehat{X}_{T_{i-1}}) \left( \Delta T_i\right)^{-1/2}\left( \prod_{D \in w} 
         \partial^{D \cap \bm v}_{\bm z^{\tau}} 
               \partial^{D \cap \bm u^W}_{\bm z^W}   (\widehat{X}_{T_{i}}) - \prod_{D \in w} 
         \partial^{D \cap \bm v}_{\bm z^{\tau}} 
               \partial^{D \cap \bm u^W}_{\bm z^W}   \widehat{X}_{T_{i-1}}\right) .
\end{aligned}
\end{equation}

    For the first term in \eqref{eq:diff_I_decoupled}, Lemma \ref{lemma: deri x growth} ensures that $\partial^{D \cap \bm v}_{\bm z^{\tau}} \partial^{D \cap \bm u^W}_{\bm z^W} \widehat{X}_{T_i} \in \mathcal{C}_p(\R^{2n+2})$. Thus, it suffices to bound the function difference. Let the state increment be defined as:
\begin{equation}\label{eq:increment_X}
    \Delta X_i = b(\widehat{X}_{T_{i-1}})\Delta T_i + \sigma_0\sqrt{\Delta T_i}\ z_i^W,
\end{equation}
    such that $\widehat{X}_{T_i} = \widehat{X}_{T_{i-1}} + \Delta X_i$. Since the derivatives of $f$ are bounded, $f^{(|w|)}$ is Lipschitz continuous. Consequently:
    \begin{equation}
    \begin{aligned}
        &\left|\left( \Delta T_i\right)^{-1/2}\left( f^{(|w|)}(\widehat{X}_{T_{i}}) - f^{(|w|)}(\widehat{X}_{T_{i-1}}) \right)\right| \\
        &= \left( \Delta T_i\right)^{-1/2}\left| f^{(|w|)}(\widehat{X}_{T_{i-1}}+\Delta X_i) - f^{(|w|)}(\widehat{X}_{T_{i-1}}) \right|\\
        &\le C\left( \Delta T_i\right)^{-1/2}\left| \Delta X_i\right|\\
        & \le  C\left( \Delta T_i\right)^{-1/2}\left| b(\widehat{X}_{T_{i-1}})\Delta T_i + \sigma_0\sqrt{\Delta T_i}\ z_i^W\right|\\
        &\le C(1 + |\widehat{X}_{T_{i-1}}|\sqrt{\Delta T_i} + \sigma_0z_i^W),
    \end{aligned}
    \end{equation}
    Applying Lemma \ref{lemma: deri x growth}, this term clearly exhibits polynomial growth.

    For the second term in \eqref{eq:diff_I_decoupled}, since $f^{(|w|)}$ is bounded, it suffices to verify that
 \begin{equation}\label{eq:second_term_expansion}
    \begin{aligned}
    &\left( \Delta T_i\right)^{-1/2}\left| \prod_{D \in w} \partial^{D \cap \bm v}_{\bm z^{\tau}} \partial^{D \cap \bm u^W}_{\bm z^W}  (\widehat{X}_{T_{i}}) - \prod_{D \in w} \partial^{D \cap \bm v}_{\bm z^{\tau}} \partial^{D \cap \bm u^W}_{\bm z^W}   \widehat{X}_{T_{i-1}}\right| \\
    & = \left( \Delta T_i\right)^{-1/2}\left| \prod_{D \in w} \partial^{D \cap \bm v}_{\bm z^{\tau}} \partial^{D \cap \bm u^W}_{\bm z^W}  (\widehat{X}_{T_{i-1}}+\Delta X_i) - \prod_{D \in w} \partial^{D \cap \bm v}_{\bm z^{\tau}} \partial^{D \cap \bm u^W}_{\bm z^W}   \widehat{X}_{T_{i-1}}\right|\\
    & = \left( \Delta T_i\right)^{-1/2}\left| \sum_{\emptyset \neq S \subseteq w} \left( \prod_{D \in S} \partial^{D \cap \bm v}_{\bm z^{\tau}} \partial^{D \cap \bm u^W}_{\bm z^W} \Delta X_i \right) \left( \prod_{D \in w \setminus S} \partial^{D \cap \bm v}_{\bm z^{\tau}} \partial^{D \cap \bm u^W}_{\bm z^W} \widehat{X}_{T_{i-1}} \right) \right|\\
    &\le \sum_{\emptyset \neq S \subseteq w}\left( \Delta T_i\right)^{-1/2}\left| \prod_{D \in S} \partial^{D \cap \bm v}_{\bm z^{\tau}} \partial^{D \cap \bm u^W}_{\bm z^W} \Delta X_i \right|\left| \prod_{D \in w \setminus S} \partial^{D \cap \bm v}_{\bm z^{\tau}} \partial^{D \cap \bm u^W}_{\bm z^W} \widehat{X}_{T_{i-1}} \right|.
    \end{aligned}
\end{equation}
    By Lemma \ref{lemma: deri x growth}, the derivatives of $\widehat{X}_{T_{i-1}}$ belong to $\mathcal{C}_p(\R^{2n+2})$. Thus, it remains to bound $\left( \Delta T_i\right)^{-1/2} \prod_{D \in S} |\partial^{D \cap \bm v}_{\bm z^{\tau}} \partial^{D \cap \bm u^W}_{\bm z^W} \Delta X_i|$. For any multi-indices $\bm v^{\tau}, \bm v^W$ such that $\bm v^{\tau} \cup \bm v^W \neq \emptyset$, we have:
\begin{equation}
\begin{aligned}
    \left|\partial^{\bm v^{\tau}}_{\bm z^{\tau}}\partial^{\bm v^{W}}_{\bm z^{W}} \Delta X_i\right|  &= \left|\partial^{\bm v^{\tau}}_{\bm z^{\tau}}\partial^{\bm v^{W}}_{\bm z^{W}}\left(b(\widehat{X}_{T_{i-1}})\Delta T_i + \sigma_0\sqrt{\Delta T_i}\ z_i^W \right)\right|\\
    & \le\sum_{\bm v_1^{\tau} + \bm v_2^{\tau} = \bm v^{\tau}}\left|\partial^{\bm v_1^{\tau}}_{\bm z^{\tau}}\partial^{\bm v^W}_{\bm z^W} b(\widehat{X}_{T_{i-1}})\right|\left| \partial^{\bm v_2^{\tau}}_{\bm z^{\tau}}\Delta T_i\right| +  \left|\sigma_0\partial^{\bm v^W}_{\bm z^W} z_i^W \partial^{\bm v^{\tau}}_{\bm z^{\tau}}\sqrt{\Delta T_i}\right|\\
    &\le C(\Delta T_i) \sum_{\bm v_1^{\tau} + \bm v_2^{\tau} = \bm v^{\tau}}\left|\partial^{\bm v_1^{\tau}}_{\bm z^{\tau}}\partial^{\bm v^W}_{\bm z^W} b(\widehat{X}_{T_{i-1}})\right| \prod_{j\in \bm v_2^{\tau}} |z_j^{\tau}|\\
    &\qquad+ C \sum_{w\in \pi(\bm v^{\tau})} (\Delta T_k)^{\frac{1}{2}}  \prod_{\lambda \in w}\prod_{i \in \lambda} |z_i^{\tau}|\left|\sigma_0\partial^{\bm v^W}_{\bm z^W} z_i^W  \right|,
\end{aligned}
\end{equation}
    where the final inequality follows from Lemma \ref{lemma: deri delta t growth} and \eqref{eq:sqrt_derivative}. Given that $\partial^{\bm v_1^{\tau}}_{\bm z^{\tau}}\partial^{\bm v^W}_{\bm z^W} b(\widehat{X}_{T_{i-1}}) \in \mathcal{C}_p(\R^{2n+2})$ (by Lemma \ref{lemma: deri x growth} and Assumption \ref{assum: smooth drift}), substituting this bound into \eqref{eq:second_term_expansion} confirms that the second term of \eqref{eq:diff_I_decoupled} also belongs to $\mathcal{C}_p(\R^{2n+2})$.
    This completes the proof.
\end{proof}


\begin{thebibliography}{10}

\bibitem{andersson2017Unbiased}
{\sc P.~Andersson and A.~{Kohatsu-Higa}}, {\em Unbiased simulation of stochastic differential equations using parametrix expansions}, Bernoulli, 23 (2017).

\bibitem{beskos2006Retrospective}
{\sc A.~Beskos, O.~Papaspiliopoulos, and G.~O. Roberts}, {\em Retrospective exact simulation of diffusion sample paths with applications}, Bernoulli, 12 (2006).

\bibitem{beskos2005Exact}
{\sc A.~Beskos and G.~O. Roberts}, {\em Exact simulation of diffusions}, The Annals of Applied Probability, 15 (2005), \url{https://arxiv.org/abs/math/0602523}.

\bibitem{caflisch1997Valuation}
{\sc R.~E. Caflisch, W.~Morokoff, and A.~Owen}, {\em Valuation of mortgage-backed securities using {{Brownian}} bridges to reduce effective dimension}, Journal of Computational Finance, 1 (1997), pp.~27--46.

\bibitem{david2003orderstatistics}
{\sc H.~A. David and H.~N. Nagaraja}, {\em Order Statistics}, Wiley, Hoboken, NJ, 3~ed., 2003.

\bibitem{dick2010}
{\sc J.~Dick and F.~Pillichshammer}, {\em Digital Nets and Sequences. Discrepancy Theory and Quasi–Monte Carlo Integration}, Cambridge University Press, Cambridge, 2010.

\bibitem{doumbia2017Unbiased}
{\sc M.~Doumbia, N.~Oudjane, and X.~Warin}, {\em Unbiased {{Monte Carlo}} estimate of stochastic differential equations expectations}, ESAIM: Probability and Statistics, 21 (2017), pp.~56--87.

\bibitem{gallager1996discrete}
{\sc R.~G. Gallager}, {\em Discrete Stochastic Processes}, Kluwer Academic Publishers, 1996.

\bibitem{giles2009Multilevela}
{\sc M.~B. Giles and B.~J. Waterhouse}, {\em Multilevel quasi-{{Monte Carlo}} path simulation}, in Advanced {{Financial Modelling}}, H.~Albrecher, W.~J. Runggaldier, and W.~Schachermayer, eds., Walter de Gruyter, Oct. 2009, pp.~165--182.

\bibitem{glasserman2004}
{\sc P.~Glasserman}, {\em {M}onte {C}arlo Methods in Financial Engineering}, Springer, New York, 2004.

\bibitem{he2014Good}
{\sc Z.~He and X.~Wang}, {\em Good path generation methods in quasi-{{Monte Carlo}} for pricing financial derivatives}, SIAM Journal on Scientific Computing, 36 (2014), pp.~B171--B197.

\bibitem{he2015convergence}
{\sc Z.~He and X.~Wang}, {\em On the convergence rate of randomized quasi--{{Monte Carlo}} for discontinuous functions}, SIAM Journal on Numerical Analysis, 53 (2015), pp.~2488--2503.

\bibitem{he2015}
{\sc Z.~He and X.~Wang}, {\em On the convergence rate of randomized quasi--{M}onte {C}arlo for discontinuous functions}, SIAM Journal on Numerical Analysis, 53 (2015), pp.~2488--2503.

\bibitem{henry-labordere2017Unbiased}
{\sc P.~{Henry-Labord{\`e}re}, X.~Tan, and N.~Touzi}, {\em Unbiased simulation of stochastic differential equations}, The Annals of Applied Probability, 27 (2017).

\bibitem{Hlawka1972}
{\sc E.~Hlawka and R.~Mück}, {\em Über eine transformation von gleichverteilten folgen {II}}, Computing, 9 (1972), pp.~127--138.

\bibitem{kloeden1992numerical}
{\sc P.~E. Kloeden and E.~Platen}, {\em Numerical {{Solution}} of {{Stochastic Differential Equations}}}, Springer Berlin Heidelberg, Berlin, Heidelberg, 1992.

\bibitem{kuo2016Application}
{\sc F.~Y. Kuo and D.~Nuyens}, {\em Application of quasi-{{Monte Carlo Methods}} to elliptic {{PDEs}} with random diffusion coefficients: A survey of analysis and implementation}, Foundations of Computational Mathematics, 16 (2016), pp.~1631--1696.

\bibitem{kuo2010}
{\sc F.~Y. Kuo, I.~H. Sloan, G.~W. Wasilkowski, and B.~J. Waterhouse}, {\em Randomly shifted lattice rules with the optimal rate of convergence for unbounded integrands}, Journal of Complexity, 26 (2010), pp.~135--160.

\bibitem{kuo2006a}
{\sc F.~Y. Kuo, G.~W. Wasilkowski, and B.~J. Waterhouse}, {\em Randomly shifted lattice rules for unbounded integrands}, Journal of Complexity, 22 (2006), pp.~630--651.

\bibitem{mcleish2011general}
{\sc D.~McLeish}, {\em A general method for debiasing a {{Monte Carlo}} estimator}, Monte Carlo Methods and Applications, 17 (2011).

\bibitem{Niederreiter1992}
{\sc H.~Niederreiter}, {\em Random Number Generation and Quasi-Monte Carlo Methods}, SIAM, Philadelphia, PA, 1992.

\bibitem{ouyang2024achieving}
{\sc D.~Ouyang, X.~Wang, and Z.~He}, {\em Achieving high convergence rates by quasi-monte carlo and importance sampling for unbounded integrands}, SIAM Journal on Numerical Analysis, 62 (2024), pp.~2393--2414.

\bibitem{owen1995}
{\sc A.~B. Owen}, {\em Randomly permuted (t,m,s)-nets and (t, s)-sequences}, in Monte Carlo and Quasi-Monte Carlo Methods in Scientific Computing, H.~Niederreiter and P.~J.-S. Shiue, eds., New York, NY, 1995, Springer New York, pp.~299--317.

\bibitem{owen1997b}
{\sc A.~B. Owen}, {\em Scrambled net variance for integrals of smooth functions}, The Annals of Statistics, 25 (1997), pp.~1541--1562.

\bibitem{owen2006a}
{\sc A.~B. Owen}, {\em Halton sequences avoid the origin}, SIAM Review, 48 (2006), pp.~487--503.

\bibitem{owen2008}
{\sc A.~B. Owen}, {\em Local antithetic sampling with scrambled nets}, The Annals of Statistics, 36 (2008), pp.~2319--2343.

\bibitem{owen2013}
{\sc A.~B. Owen}, {\em Monte Carlo Theory, Methods and Examples}, Stanford, 2013.

\bibitem{pyke1965spacings}
{\sc R.~Pyke}, {\em Spacings}, Journal of the Royal Statistical Society. Series B (Methodological), 27 (1965), pp.~395--449.

\bibitem{rhee2015Unbiased}
{\sc C.-H. Rhee and P.~W. Glynn}, {\em Unbiased estimation with square root convergence for {{SDE}} models}, Operations Research, 63 (2015), pp.~1026--1043.

\bibitem{sloan1998When}
{\sc I.~H. Sloan and H.~Wo{\'z}niakowski}, {\em When are quasi-{Monte Carlo} algorithms efficient for high dimensional integrals?}, Journal of Complexity, 14 (1998), pp.~1--33.

\bibitem{song2021ScoreBased}
{\sc Y.~Song, J.~{Sohl-Dickstein}, D.~P. Kingma, A.~Kumar, S.~Ermon, and B.~Poole}, {\em Score-based generative modeling through stochastic differential equations}, Feb. 2021, \url{https://doi.org/10.48550/arXiv.2011.13456}, \url{https://arxiv.org/abs/2011.13456}.

\bibitem{vihola2018Unbiased}
{\sc M.~Vihola}, {\em Unbiased estimators and multilevel monte carlo}, Operations Research, 66 (2018), pp.~448--462.

\bibitem{wang2020a}
{\sc H.~Wang, T.~Zariphopoulou, and X.~Y. Zhou}, {\em Reinforcement learning in continuous time and space: A stochastic control approach}, Journal of Machine Learning Research, 21 (2020), pp.~1--34.

\bibitem{wang2006Effects}
{\sc X.~Wang}, {\em On the effects of dimension reduction techniques on some high-dimensional problems in finance}, Operations Research, 54 (2006), pp.~1063--1078.

\bibitem{wang2009Dimension}
{\sc X.~Wang}, {\em Dimension reduction techniques in quasi-{Monte Carlo} methods for option pricing}, INFORMS Journal on Computing,  (2009).

\bibitem{wang2016Handling}
{\sc X.~Wang}, {\em Handling discontinuities in financial engineering: Good path simulation and smoothing}, Operations Research, 64 (2016), pp.~297--314.

\bibitem{wang2003effectivea}
{\sc X.~Wang and K.-T. Fang}, {\em The effective dimension and quasi-{{Monte Carlo}} integration}, Journal of Complexity, 19 (2003), pp.~101--124.

\bibitem{wang2010QuasiMonte}
{\sc X.~Wang and I.~H. Sloan}, {\em Quasi-{{Monte Carlo}} methods in financial engineering: An equivalence principle and dimension reduction}, Operations Research,  (2010).

\bibitem{wang2012Pricing}
{\sc X.~Wang and K.~S. Tan}, {\em Pricing and hedging with discontinuous functions: Quasi--{{Monte Carlo}} methods and dimension reduction}, Management Science,  (2012).

\bibitem{YZbook}
{\sc J.~Yong and X.~Y. Zhou}, {\em Stochastic {C}ontrols: {H}amiltonian {S}ystems and {HJB} {E}quations}, New York, NY: Spinger, 1999.

\end{thebibliography}
\end{document}